\documentclass{amsart}

\usepackage[letterpaper,margin=1in]{geometry}

\usepackage{amsmath,amssymb,amsthm}
\usepackage{mathabx}
\usepackage{hyperref}
\usepackage{mathtools}
\usepackage{tikz-cd}
\usepackage{enumerate}

\newtheorem{introthm}{Theorem}

\newtheorem{theorem}{Theorem}[section]
\newtheorem{proposition}[theorem]{Proposition}
\newtheorem{lemma}[theorem]{Lemma}
\newtheorem{corollary}[theorem]{Corollary}

\theoremstyle{definition}
\newtheorem{definition}[theorem]{Definition}

\theoremstyle{remark}
\newtheorem{remark}[theorem]{Remark}
\newtheorem{example}[theorem]{Example}

\numberwithin{equation}{section}

\newcommand{\R}{\mathbb{R}}
\newcommand\C{\mathbb{C}}
\newcommand\N{\mathbb{N}}

\newcommand{\Cst}{\mathrm{C}^*}

\newcommand{\cC}{\mathcal{C}}

\DeclareMathOperator{\tr}{tr}

\DeclareMathOperator{\id}{id}

\DeclareMathOperator{\Span}{Span}

\DeclareMathOperator{\inv}{inv}
\DeclareMathOperator{\Lip}{Lip}
\DeclareMathOperator{\dom}{dom}
\DeclareMathOperator{\sa}{sa}
\DeclareMathOperator{\ev}{ev}
\DeclareMathOperator{\sgn}{sgn}

\DeclarePairedDelimiter{\ip}{\langle}{\rangle}

\DeclarePairedDelimiter{\norm}{\lVert}{\rVert}

\title{Compact quantum metric spaces from free probability}

\author{David Jekel}
\address{\parbox{\linewidth}{Department of
		Mathematical Sciences, University of Copenhagen \\
		Universitetsparken 5, 2100 Copenhagen \O, Denmark}}
\email{daj@math.ku.dk}

\author{Therese Basa Landry}
\address{\parbox{\linewidth}{Department of
		Mathematics, South Hall, Room 6607, University of California, \\
		Santa Barbara, CA 93106, United States}}
\email{tlandry@ucsb.edu}


\begin{document}
	
	\maketitle
	
	\begin{abstract}
		We study quantum metric space structures on operator algebras arising from free probability, namely those associated to $q$-Gaussians and free Gibbs laws for convex potentials.  We note that even for free semicirculars, Voiculescu's dual system does not produce a quantum metric space structure that recovers the weak-$*$ topology on the state space.  However, for $q$-Gaussians, we can define a compact quantum metric space using length-like functions by the same method as has already been used for hyperbolic groups, quantum groups of rapid decay, free products, and free graph algebras.  Next, motivated by the free transport results for free Gibbs laws, we describe a universal way of defining Lip-norms in terms of a generating set, which behaves well under changes of coordinates.  We show using semigroup regularization that this Lip-norm defines a quantum metric space structure for $q$-Gaussians, and then transfer this property to free Gibbs laws for convex potentials using free transport.
	\end{abstract}
	
	\section{Introduction}
	
	The goal of this work is to describe quantum metric space structures for operator algebras arising from free probability, namely $q$-Gaussians and free Gibbs laws.  Connes' programme of non-commutative geometry aims to study analogs of differential and Riemannian structure in a non-commutative setting \cite{Con89,Connes1994}.  In particular, the algebra of continuous functions $C(X)$ on a Riemannian manifold $X$ is replaced by a general $\mathrm{C}^*$-algebra $A$.  The differential structure is often encoded in an unbounded Fredholm module, or spectral triple.  This consists of a concrete representation of $A$ on a Hilbert space $H$ and a Dirac operator $D$ on $H$ such that $[D,a]$ is a bounded operator for $a$ in a dense subset of $A$.  The operator $[D,a]$ is analogous to a derivative of $a$, so $\norm{a}_{\Lip} := \norm{[D,a]}$ is an analog of the Lipschitz seminorm.  Under the natural assumption that the Lipschitz seminorm vanishes only on scalar multiples of the identity, we can dualize the Lipschitz seminorm to obtain a metric on the state space $S(A)$, given by
	\[
	\operatorname{mk}(\varphi,\psi) = \sup \{ |\varphi(a) - \psi(a)|: \norm{a}_{\Lip} \leq 1\},
	\]
	which is the analog of the $L^1$-Wasserstein distance as evaluated via Monge--Kantorovich--Rubinstein duality.  Building on the work of Connes, Rieffel formalized desirable properties of generalized Lipschitz seminorms, or Lip-norms, into the notion of a compact quantum metric space, and studied the corresponding Monge--Kantorovich distances \cite{Rieffel, Rieffel2,Rieffel4}.  In particular, he characterized when $\operatorname{mk}$ produces the weak-$*$ topology on $S(A)$; given a state $\varphi_0$ on $A$, one needs to check that $\{a \in A: \varphi_0(a) = 0, \norm{a}_{\Lip} \leq 1\}$ is totally bounded with respect to operator norm \cite[Theorem 1.8]{Rieffel}.  As Rieffel also considered Lip-norms which do not necessarily come from Dirac operators, spectral triples which give rise to compact quantum metric spaces are sometimes called \textit{spectral metric spaces} \cite{AguilarKaad2018,AustadKaadKyed2025,BeMaRe,HawkZach2017,KaadMikkelsen2024,KessSam2013,Klisse2024}.
	
	It seems natural to view objects from free probability in the quantum metric space framework, since free probability already has analogs of several of the main ingredients.  Indeed, the role of derivative is played by Voiculescu's free difference quotients, and related notion of dual and conjugate systems \cite{VoiculescuFE5}.  There is also a Wasserstein distance in the free probability setting, introduced by Biane and Voiculescu \cite{BV2001}, which defines a metric on the space of tracial states for a certain universal $\mathrm{C}^*$-algebra $C([-R,R])^{* d}$.  However, the focus of the Wasserstein distance in the free setting is somewhat different.  Quantum metric spaces construct a distance on $S(A)$ for a fixed $\mathrm{C}^*$-algebra $A$.  The free Wasserstein distance compares tracial states associated to possibly different von Neumann algebras $(M_1,\tau_1)$ and $(M_2,\tau_2)$ with diferent generators $(x_1,\dots,x_d)$ and $(y_1,\dots,y_d)$ by lifting both to traces on the universal free product $\mathrm{C}^*$-algebra $C([-R,R])^{* d}$ \cite[\S 1.2]{BV2001}.  The Wasserstein topology on $\mathcal{T}(C([-R,R])^{*d})$ also turns out to be much stronger than the weak-$*$ topology \cite[\S 5]{GangboJekelNamShlyakhtenko2022}, which is markedly different from what happens in the classical setting, but perhaps not surprising for such a large $\mathrm{C}^*$-algebra.
	
	Our present goal is define a Lip-norm on certain $\mathrm{C}^*$-algebras coming from free probability.  In particular, the prototypical example is the $\mathrm{C}^*$-algebra $A$ generated by a family of freely independent semicircular operators $x = (x_1, \dots, x_d)$.  This comes with a canonical trace $\tau$ and sits inside the tracial von Neumann algebra $M$, which is a isomorphic to the free group von Neumann algebra $L(F_d)$.  We view $A$ as acting on the Hilbert space $L^2(A,\tau)$.  This serves as a noncommutative analog of the coordinate-multiplication operators $x_j$ on $L^2(\R^d,\gamma)$ where $\gamma$ is the standard Gaussian measure.  In the classical setting, $i d/dx_j$ defines an unbounded operator on $L^2(\R^d,\gamma)$, which also satisfies $[d/dx_j, f(x)] = \partial_j f(x)$ as operators on $L^2(\R^d,\gamma)$, and in particular $[d/dx_j,x_k] = \delta_{j=k}$.
	
	Going back to the free setting, Voiculescu defined a \emph{dual system} as a family of operators $(Y_1,\dots,Y_d)$ on $L^2(A,\tau)$ satisfying $[Y_j, x_k] = \delta_{j=k} P_\Omega$, where $\Omega = \widehat{1}$ is the vacuum vector in the GNS Hilbert space \cite[Definition 5.1]{VoiculescuFE5}; see also \cite{Shlyakhtenko2006}.  One could first attempt to define a Lip-norm on $a$ by $\max_j \norm{[Y_j,a]}$, but this immediately fails to satisfy Rieffel's total boundedness condition.  Indeed, the operators $Y_j$ in the semicircular case turn out to be bounded operators \cite[Example 5.13]{VoiculescuFE5}, and so $\{a \in A: \tau(a) = 0, \norm{a}_{\Lip} \leq 1\}$ contains $\{a \in A: \tau(a) = 0, \norm{a} \leq r\}$ for some $r > 0$, hence is not totally bounded.  However, given that $[Y_j, x_k] = \delta_{j=k} P_\Omega$, we have that $[Y_j,f(x)]$ is a compact and in fact Hilbert--Schmidt operator for all polynomials $f$.  Thus, we could instead try $\max_j \norm{[Y_j,a]}_{HS}$ as a candidate for the Lipschitz norm.
	
	However, not even the Hilbert--Schmidt norm of $[Y_j,a]$ fulfills the total boundedness condition.  To see this, we need to describe $Y_j$ more explicitly in terms of Voiculescu's free difference quotient.  Letting $\C\ip{x}$ denote the space of polynomials in $x$, the difference quotient $\partial_j$ is a derivation $\C\ip{x} \to \C\ip{x} \otimes \C\ip{x} \subseteq L^2(A,\tau) \otimes L^2(A,\tau)$ satisfying $\partial_j (x_k) = \delta_{j=k} 1 \otimes 1$ \cite[\S 3]{VoiculescuFE5}.  There is a natural identification of $L^2(A,\tau) \otimes L^2(A,\tau)$ with the space of Hilbert--Schdmidt operators, where $a \otimes b$ corresponds to the operator $h \mapsto a \tau(bh)$.  This correspondence identifies the Hilbert--Schmidt operator $[Y_j,f(x)]$ with the tensor $\partial_j f(x)$; see \cite[Proposition 5.11]{VoiculescuFE5}.  We now consider the $n$th Chebyshev polynomial $g_n$ of the second kind, which satisfies $g_n(2 \cos \theta) = \sin((n+1)\theta) / \sin(\theta)$; the normalization here is chosen so that $g_n$ is defined on $[-2,2]$ and forms an orthonormal family with respect to the semicircular measure $(1/2\pi) \sqrt{4 - x^2}\,dx$; see e.g.\ \cite[p. 159]{MingoSpeicher2017}, \cite[\S 3]{Parraud2025Haagerup}.  It turns out that $\partial_j g_n(x_j) = \sum_{k=0}^{n-1} g_k \otimes g_{n-1-k}$ and $\partial_{j'} g_n(x_j) = 0$ for $j' \neq j$; see e.g.\ \cite[\S 2.3]{MingoSpeicher2013}.  We therefore have
	\begin{align*}
		\tau(g_n(x_k)) &= \delta_{n=0}, \\
		\norm{g_n(x_j)}_{L^2(A,\tau)} &= 1, \\
		\norm{g_n(x_j)}_A &= n + 1, \\
		\norm{\partial_j g_n(x_j)}_{L^2(A,\tau) \otimes L^2(A,\tau)} &= n^{1/2}.
	\end{align*}
	Hence, $\norm{g_n(x_j)}_A$ is not even \emph{bounded} by a constant times $\norm{\partial_j g_n(x_j)}_{L^2(A,\tau) \otimes L^2(A,\tau)} = \norm{[Y_j,g_n(x_j)]}_{HS}$, and so $\{f: \norm{\partial_j f}_{L^2(A,\tau) \otimes L^2(A,\tau)} \leq 1, \tau(f) = 0\}$ is not bounded in operator norm, let alone totally bounded.
	
	In short, bounds on the $L^2$-norm of free difference quotient are much weaker than bounds on the $L^2$ norm of the derivative in the classical setting.  We can seek to remedy this in two ways, either by putting a stronger norm on the tensor product (such as an operator norm), or by replacing the dual system $Y_j$ with a stronger differential operator, for instance a second-order differential operator.  The latter approach is the one we take here.  We remark that for free semicircular variables, $\sum_{j=1}^d \partial_j^* \partial_j$ turns out to be an operator that multiplies the $n$th order orthogonal polynomials in the semicircular family $n$.  It thus provides a natural analog of the length function on the free group.
	
	Length functions on groups have in fact been well-studied as a means of producing spectral triples \cite{AntonescuChristiansen2004,ChristRieffel2017,Con89,FarLandLarsPack2024,Klisse2024,OzRi2005hyperbolic,Rieffel3}.  Spectral metric spaces constructed from length functions have also been developed for quantum groups \cite{AustadKyed2026}. Suppose we are given a countable group $G$ and a length function $\ell: G \to \N$.  Consider the left regular representation of $G$ on $\ell^2(G)$, which generates the reduced group $\mathrm{C}^*$-algebra.  A Dirac operator $D$ on $\ell^2(G)$ can be defined as the diagonal operator of multiplication by $\ell$.  Ozawa and Rieffel \cite[Main Theorem 1.2]{OzRi2005hyperbolic} gave a sufficient condition to check whether the length spectral triple satisfies Rieffel's total boundedness condition, which they used to study hyperbolic groups \cite[\S 4]{OzRi2005hyperbolic}; they also showed that the hypotheses of their theorem are preserved under reduced free products \cite[\S 6]{OzRi2005hyperbolic}.  Their criterion has also been used by Aguilar, Hartglass, and Penneys to study spectral triples on free graph algebras and $\mathrm{C}^*$-algebra generated by Guionnet--Jones--Shlyakhtenko planar algebras \cite{AguilarHartglassPenneys2022}, which in particular includes the case of a free semicircular family.
	
	Our first main result applies Ozawa and Rieffel's criterion to study the length-like spectral triple for $q$-Gaussians.  The $q$-Gaussians for $q \in [-1,1]$, first studied in \cite{BoKuSp1997qgaussian,BoSp1991example,BoSp1996interpolations}, are a natural interpolation between a classical Gaussian family ($q = 1$), a free semicircular family ($q = 0$), and a fermionic family ($q = -1$).  This family of examples has inspired a long line of research, and for $q \in (-1,1)$, they have many similar properties to the free semicircular family (see e.g. \cite{MiyagawaSpeicher2023}), through often the proofs are much more challenging.  The $q$-Gaussians are constructed as operators on a certain Fock space, with the tracial state being realized by the vacuum vector (see \S \ref{subsec: q Gaussian setup}); the $n$th level of the Fock space corresponds to Wick words of length $n$, and this provides a natural filtration of $A$ in the sense of Ozawa and Rieffel.  We can then check Ozawa and Rieffel's conditions using Bo{\.z}ejko's Haagerup-like inequalities \cite{Bozejko1999ultracontractivity}, which leads to the following result.
	
	\begin{introthm} \label{thm: q Gaussian length}
		Let $q \in [-1,1)$.  Let $\mathcal{A}_q(H_{\R})$ be a $q$-Gaussian algebra over a finite-dimensional real Hilbert space $H_{\R}$ with ONB $(e_i)_{i \in I}$, canonical generators $(x_i)_{i \in I}$ corresponding to the basis elements, and $\pi$ its canonical representation on the $q$ Fock space $\mathcal{F}_q(H)$.  Let $P_n$ be the projection onto the $n$th summand $H^{\otimes n}$ in the $q$-Gaussian Fock space $\mathcal{F}_q(H)$ and $D_q = \sum_{n=1}^\infty n P_n$ as above.  Then $(\mathcal{A}_q(H_{\R}),\mathcal{F}_q(H),D_q)$ with representation $\pi$ is a spectral triple.  Furthermore, this spectral triple is also a spectral metric space on $\mathcal{A}_q(H_{\R})$.
	\end{introthm}
	
	In the free case $q = 0$, the operator $D_q$ is the ``free Laplacian'' $\sum_{j=1}^d \partial_j^* \partial_j$, so this result is covered by \cite{AguilarHartglassPenneys2022}.  Meanwhile, at $q = 1$, it is the classical Laplacian $\Delta$, and for general $q$, the operator corresponds to a certain $q$-Laplacian \cite[Proposition 30]{Dabrowski2014}.  Theorem \ref{thm: q Gaussian length} does not hold at $q = 1$ because the commutator $[\Delta, f(x)]$ fails to be bounded for polynomial $f$ since $[\Delta, f] g = \Delta(fg) - f \Delta g = (\Delta f) g + \ip{\nabla f, \nabla g}$, and thus, for instance, $[\Delta, x_j] = \partial_j$.  Hence, in the classical setting $\nabla$ or $|\nabla|$ would be more natural choices for a Dirac operator, while in the free setting the Laplacian is more effective at defining a quantum metric structure.  We also note that at $q = -1$, the algebra $\mathcal{A}_q(H_{\R})$ becomes finite-dimensional, and there is a well-known isomorphism between the $d$-variable fermionic Fock space and the space generated by $d$ qubits (see \cite[\S 2.2]{CarlenMaas2014analog}).  In \S \ref{subsec: fermionic}, we explicitly describe the fermionic Gaussian variables and the commutator with the operator $D_{-1}$ in terms of tensors of Pauli matrices acting on $(\C^2)^{\otimes d}$.
	
	Returning again to the free setting, we would like to know whether  $\sum_{j=1}^d \partial_j^* \partial_j$ can be used to define a spectral triple for more general families of operators $x_j$'s arising in free probability.  Unfortunately, very little is known about the finer properties of the free Laplacian.  For instance, even under reasonable free analytic assumptions such as finite free Fisher information or Lipschitz conjugate variables, it is unknown whether $\sum_{j=1}^d \partial_j^* \partial_j$ has compact resolvent in general, how to describe its eigenvectors, nor whether it has good properties relating the operator norm and $2$-norm in line with the rapid decay and ultracontractivity results that hold for $q$-Gaussians \cite{Bozejko1999ultracontractivity}.
	
	One case that appears the most tractable is that of free Gibbs laws associated with a convex potential $V$ \cite{GMS2006,GS2009,Jekel2020Entropy}.  These are the analogs of classical log-concave measures with density $e^{-V}$ for some convex $V$, and the free Gibbs laws in fact arise from the large-$n$ behavior of $n \times n$ matrix models with a log-concave density on $M_n(\mathbb{C})_{\sa}^d$ of the form $e^{-n^2 V}$ where $V = \tr(p)$ for some non-commutative polynomial, or more generally $V$ is a non-commutative power series or smooth function.  We know that under appropriate hypotheses on $V$, there exists a non-commutative smooth change of coordinates, or transport map, that transforms the tuple $(x_1,\dots,x_d)$ given by the free Gibbs law into a free semicircular family $(z_1,\dots,z_d)$.  Such maps were constructed by Guionnet and Shlyakhtenko \cite{GS2014} using an analog of results from optimal transport theory, and further generalizations and related results are given in \cite{Nelson2015transport,DGS2021,Jekel2022Expectation,JekelLiShlyakhtenko2022}.  However, even with the cheat code of transport, we do not immediately see whether the property of defining a Dirac operator for a spectral metric space passes from the free Laplacian in the one coordinate system to the free Laplacian in other, as this would require controlling the extra terms that arise from application of the chain rule for the free difference quotients when computing the commutator with the free Laplacian.
	
	We therefore present another approach to define Lip-norms that behave well under the types of coordinate changes that occur in free transport theory.  In fact, this is a general construction of Lip-norms that can be applied to any given self-adjoint tuple $(x_1,\dots,x_d)$ generating a unital $\mathrm{C}^*$-algebra $A$, and hence may be of wider interest for constructing quantum metric spaces.  For $f \in A$, define
	\begin{equation} \label{eq: universal Lipschitz norm}
		L_x(f) = \sup_{(\tilde{A}, \alpha, \beta)} \frac{\norm{\alpha(f) - \beta(f)}_{\tilde{A}}}{\max_j \norm{\alpha(x_j) - \beta(x_j)}_{\tilde{A}}},
	\end{equation}
	where $(\tilde{A},\alpha,\beta)$ ranges over unital $\mathrm{C}^*$-algebras $\tilde{A}$ and $*$-homomorphisms $\alpha, \beta: A \to \tilde{A}$.  Thus, the Lipschitz norm of $f$ measures how far the images of $f$ can differ under two $*$-homomorphisms compared with how much the images of the generators differ.  As motivation for this, consider that if $A$ is commutative, one can take $\tilde{A} = \mathbb{C}$ and the homomorphisms $\alpha$, $\beta$ to be point evaluations.
	
	The definition \eqref{eq: universal Lipschitz norm} is also quite flexible, since we can specify a subclass $\mathcal{C}$ of triples $(\tilde{A},\alpha,\beta)$ as the candidates for the supremum; we denote the resulting quantity by $L_{x,\mathcal{C}}(f)$. The subclass $\mathcal{C}$ can be chosen quite freely, as long as it contains enough nontrivial choices of differing $\alpha, \beta$.  For example, if we consider a spectral triple with Dirac operator $D$, then we can consider the set of maps $A \to B(H)$ given by $\alpha_t(a) = e^{itD} a e^{-itD}$ for $t$ in a certain interval, and then \eqref{eq: universal Lipschitz norm} results in something comparable to the spectral Lip-norm $\norm{[D,a]}_{B(H)}$.  Another natural class is $\mathcal{C}_{\tr}(A,\tau)$ for a given $\mathrm{C}^*$-algebra and faithful tracial state $\tau$, which we define as the class of triples $(\tilde{A},\alpha,\beta)$ where $\tilde{A}$ has a faithful trace $\tilde{\tau}$ and $\alpha$ and $\beta$ are trace-preserving unital $*$-homomorphisms; this is motivated by the couplings that appear in Biane and Voiculescu's definition of free Wasserstein distance \cite[\S 2.1]{BV2001}.
	
	In \S \ref{sec: couplings}, we develop several basic properties of the Lip-norms $L_{x,\mathcal{C}}$.  In particular, we show that they produce quantum metric space structures for $q$-Gaussians as well as free Gibbs laws associated to appropriate convex potentials.  We state a sample of the results here; see \S \ref{sec: couplings} for details.
	
	\begin{introthm}
		Fix a $\mathrm{C}^*$-algebra with a faithful tracial state $\tau$ generated by self-adjoints $(x_1,\dots,x_d)$.  Take $\mathcal{C} = \mathcal{C}_{\tr}(A,\tau)$, and define $L_{x,\mathcal{C}}$ by \eqref{eq: universal Lipschitz norm}.  Then
		\begin{enumerate}[(1)]
			\item The domain of $L_{x,\mathcal{C}}$ contains polynomials of $x$ (see Proposition \ref{prop: Lipschitz seminorm basic properties} (4)), and more generally non-commutative smooth functions in the sense of \cite{JekelLiShlyakhtenko2022} (see Lemma \ref{lem: smooth to Lipschitz}).
			\item $L_{x,\mathcal{C}}$ satisfies the Leibniz inequality (see Proposition \ref{prop: Lipschitz seminorm basic properties} (2)), finite diameter condition (see Proposition \ref{prop: Lipschitz seminorm basic properties} (1)), and MK duality (see Proposition \ref{prop: MK duality}).
			\item If $y$ is a tuple of self-adjoint operators which generate $A$ with $L_{x,\mathcal{C}}(y_j) < \infty$, then
			\[
			L_{x,\mathcal{C}}(f) \leq L_{y,\mathcal{C}}(f) \max_j L_{x,\mathcal{C}}(y_j)
			\]
			(see Lemma \ref{lem: change of coordinates}).
			\item $L_{x,\mathcal{C}}$ satisfies Rieffel's total boundedness condition, and thus defines a compact quantum metric space structure, in the following cases:
			\begin{itemize}
				\item If $x$ is a $q$-Gaussian $d$-tuple for $q \in [-1,1)$ (see Corollary \ref{cor: general Lip norm for q Gaussian}).
				\item If $x$ is a $d$-tuple described by a free Gibbs law for a potential $V$ satisfying the hypotheses of \cite[\S 8.1]{JekelLiShlyakhtenko2022} for instance (see Corollary \ref{cor: free Gibbs case}).
			\end{itemize}
		\end{enumerate}
	\end{introthm}
	
	The proofs of (1), (2), (3) are elementary, and in fact our motivation for Definition \eqref{eq: universal Lipschitz norm} was to arrange that (3) will be satisfied.  To prove the total boundedness condition (4) for $q$-Gaussians, we use regularization by the heat semigroup $\Phi_t = e^{-tD_q}$ (see Proposition \ref{prop: total boundedness via semigroup}).  Because the semigroup $\Phi_t$ admits a dilation to automorphisms, we can estimate $\norm{a - \Phi_t(a)}$ in terms of the difference of two $*$-homomorphisms and hence in terms of $L_{x,\mathcal{C}}(a)$; but meanwhile we can use Bo{\.z}ejko's ultracontractivity result to show that $\Phi_t$ of the unit ball is totally bounded in operator norm.  This implies the total boundedness of $\{a: L_{x,\mathcal{C}}(a) \leq 1, \tau(a) = 0\}$ by an approximation argument.  Once we have proved the case of $q$-Gaussians, and in particular a free semicircular family for $q = 0$, we can transfer the total boundedness condition to the free Gibbs laws by using the free transport results from \cite{JekelLiShlyakhtenko2022} together with item (3) on the behavior of $L_{x,\mathcal{C}}$ under change of coordinates.
	
	The paper is organized into two sections, \S \ref{sec: q Gaussian length} describes the length spectral triple for $q$-Gaussians, and \S \ref{sec: couplings} develops general Lip-norms, establishes their behavior under change of variables, and applies them to $q$-Gaussian variables and free Gibbs laws using free transport theory.
	
	\subsection*{Funding}
	
	DJ was supported by an EU Horizon Marie Sk{\l}odowska Curie Action\footnote{Views and opinions expressed are those of the author(s) only and do not necessarily reflect those of the European Union or the Research Executive Agency. Neither the European Union nor the granting authority can be held responsible for them.}, FREEINFOGEOM, grant id: 101209517.
	
	\subsection*{Acknowledgements}
	
	We thank the Institute of Pure and Applied Mathematics for hosting the authors during the long program on noncommutative optimal transport in Spring 2025, as well as providing travel funding for DJ.  Further collaboration occurred during TBL's visit to Copenhagen in July 2025 which was supported by the Marie Curie grant FREEINFOGEOM.  We thank Eric Carlen, Bhishan Jacelon, Marius Junge, Nadia Larsen, and Dimitri Shlyakhtenko for stimulating talks and conversations.
	
	\section{Length spectral triple for $q$-Gaussians} \label{sec: q Gaussian length}
	
	\subsection{Background on noncommutative geometry}
	
	We assume familiarity with the basic theory of $\mathrm{C}^*$-algebras; see e.g. \cite{Murphy1990}.  Quantum Wasserstein metrics in the sense of Connes and Rieffel grant a $\mathrm{C}^*$-algebra $A$ access to the theory of metric geometry via $S(A)$.  In \cite{Con89, Connes1994}, Connes showed that the geodesic distance $d_g$ on a compact Riemannian spin manifold $X$ can be recovered from a metric on $S(A)$ determined by the $\mathrm{C}^*$-algebra $C(X)$, the Hilbert space $H$ of $L^2$-spinor fields, and the Dirac operator $D$.  He formalized the operator algebraic elements essential to his approach in the definition of a spectral triple, which we review below.   
	
	
	\begin{definition}
		[Connes \cite{Con89, Connes1994}]
		\label{def:spectraltriple}
		A \emph{spectral triple} $(A, H, D)$ consists of a unital $\Cst$-algebra $A$, a unital faithful representation $\pi$ of $A$ on a Hilbert space $H$, and a self-adjoint operator $$D : \text{dom}(D) \subseteq H \rightarrow H$$ such that
		\begin{align*}
			& (ST1) \text{ the operator $D$ has compact resolvent $R_{\lambda}(D)=(D - \lambda I)^{-1}$, $\lambda \in \mathbb{C} \setminus \sigma(D)$}, \\
			& (ST2) \text{ the set }\\
			& \qquad \qquad \qquad \{ a \in A : [D, \pi(a)] \text{ is densely defined and extends to a bounded operator on } H \} \\
			& \qquad \text{ is a dense subset of $A$.}
		\end{align*}
		Moreover, $D$ is called a \emph{Dirac operator}. 
	\end{definition}
	
	Spectral triples generalize differential structure.  Based on considerations from quantum mechanics, the differential of $a$ in $A$ is given by the commutator $[D, \pi(a)]$.  The dense set in $A$ described in condition $(ST2)$ is therefore analogous to the dense set of $C^{\infty}$ functions in the manifold case.  The compact resolvent condition in $(ST1)$ ensures that the eigenvalues of $D$ exhibit properties that allow for the extraction of geometric information like measure and dimension from spectral data.  In the context of spectral triples and other more general settings, Rieffel developed a corresponding notion of metric spaces in the noncommutative setting by investigating when Connes' metric on $S(A)$ recovers the weak-$*$ topology.  He codified these conditions in the following definition.
	
	\begin{definition}
		[Rieffel \cite{Rieffel, Rieffel2, Rieffel4}]
		\label{def:qcms}  
		A \emph{ compact quantum metric space} $(A, L)$ is an ordered pair of a unital $\Cst$-algebra $A$ and a seminorm $L$ such that
		\begin{enumerate}[(1)]
			\item $\operatorname{dom}(L) L:= \{ a \in A: L(a) < \infty \}$ is dense in $A$.
			\item $L$ is $^{\ast}$-invariant and lower semi-continuous on $A$.
			\item $\{ a \in \operatorname{dom}(L) : L(a) = 0 \} = \mathbb{C} 1_A$.
			\item The {\it Monge-Kantorovich metric } $\operatorname{mk}_L$, defined on the state space $\mathcal{S}(A)$ by
			\[
			\operatorname{mk}_L(\varphi, \psi) = \sup \{ | \varphi(a) - \psi(a) | : a \in \operatorname{dom}(L), L(a) \leq 1 \} \text{ for } \varphi, \psi \in \mathcal{S}(A),
			\]
			metrizes the weak$^*$ topology on $\mathcal{S}(A)$ of $A$.
		\end{enumerate}
		Moreover, $L$ is called a \emph{ Lip-norm}.  Furthermore, $L$ is said to be a \emph{ Leibniz seminorm} and $(A, L)$ a \emph{ Leibniz compact quantum metric space} if, for any $a,b \in A$,
		$$L(ab) \leq L(a)\| b \|_A + \| a \|_AL(b).$$
	\end{definition}
	
	
	Note that in the case $A = C(X)$ for a compact metric space, we can take $L$ to be the usual Lipschitz seminorm, and then $\operatorname{mk}_L$ becomes the dual formulation of the $L^1$-Wasserstein distance on $\mathcal{P}(X)$.  Moreover, pure states correspond to point masses, and so the restriction of $\operatorname{mk}_L$ to pure states recovers the original metric.  Condition (4) that $\operatorname{mk}_L$ metrizes the weak-$*$ topology is usually the most subtle to check, and Rieffel gave the following equivalent condition directly in terms of the Lip-norm and operator norm.  As motivation, the reader should keep in mind that the $1$-Lipschitz functions on a compact space $X$, with a given bound at a single point, form a compact set with respect to the uniform norm.
	
	\begin{theorem}[{\cite[Theorem 1.8]{Rieffel}}] \label{thm: total boundedness criterion}
		Let $A$ be a $\mathrm{C}^*$-algebra and let $L$ satisfy (1) - (3) of Definition \ref{def:qcms}.  Let $\varphi_0 \in \mathcal{S}(A)$.  Then $L$ satisfies condition (4) of Definition \ref{def:qcms} if and only if $\{a \in A: \varphi_0(a) = 0, \norm{a}_{\Lip} \leq 1\}$ is totally bounded with respect to operator norm.
	\end{theorem}
	
	While Lip-norms determined from Dirac operators as in the definition below are always $^*$-invariant, Leibniz and lower-semicontinuous, condition (2) or condition (3) of Definition \ref{def:qcms} may not be satisfied \cite{AguilarKaad2018}.  Since our spectral triples give seminorms satisfying these conditions, we adopt the following convention.
	
	\begin{definition}
		Let $(A, H, D)$ be a spectral triple with representation $\pi$ and for $x \in A$, let  $L_{D}(x) := \|[D, \pi(x)] \|_{B(H)}$.  Then $(A, H, D)$ is said to be a \emph{spectral metric space} if $(A, L_D)$ is a compact quantum metric space. 
	\end{definition}
	
	Examples of spectral metric spaces include group $\mathrm{C}^*$-algebras \cite{ChristRieffel2017,OzRi2005hyperbolic,Rieffel3} and crossed product $\mathrm{C}^*$-algebras \cite{AustadKaadKyed2025,BeMaRe}, as well as noncommutative solenoids \cite{FarLandLarsPack2024}, Bunce-Deddens algebras, noncommutative tori, and some of their generalizations \cite{HawkSkalStuZach2013, Klisse2024}.  Spectral metric space structures have also been constructed for Gibbs measures \cite{KessSam2013}, quantum groups \cite{AustadKyed2026}, quantum projective spaces \cite{KaadMikkelsen2024}, and quantum spheres including and beyond the Podle\'{s} spheres \cite{AguilarKaad2018,HawkZach2017}.
	
	We next recall the main result of Ozawa and Rieffel in \cite{OzRi2005hyperbolic}, which gives sufficient conditions to be able to define a compact quantum metric space via a length-like spectral triple.  Their criterion is applicable to the class of $\mathrm{C}^*$-algebras described below.
	
	\begin{definition}
		Let $A$ be a unital $\Cst$-algebra.  If there exists a family of finite-dimensional $^{\ast}$-closed subspaces $\{ {A}_n \}_{n \in \N_0}$ such that
		\begin{align*}
			& \qquad \text{(1) $\bigcup_{n=0}^{\infty} A_n$ is a unital complex dense $^{\ast}$-subalgebra of $A$, } \qquad \qquad \qquad \qquad \qquad \qquad \qquad \qquad \qquad \\
			& \qquad \text{(2) ${A}_0 = \C$,} \\
			& \qquad \text{(3) ${A}_m \subseteq {A}_n$ for $m < n$,}\\
			& \qquad \text{(4) ${A}_m {A}_n \subseteq {A}_{m+n}$ for all $m,n$,}
		\end{align*}
		then $A$ is said to be a \emph{filtered $\Cst$-algebra.}
	\end{definition}
	
	Ozawa and Rieffel used Haagerup type inequalities to build spectral metric spaces for hyperbolic group $\mathrm{C}^*$-algebras.  In fact, their techniques have been adapted to produce compact quantum metric space structures in settings as various as $\mathrm{C}^*$-algebras generated by planar algebras \cite{AguilarHartglassPenneys2022}, crossed products of $\mathrm{C}^*$-algebras by actions of discrete groups \cite{HawkSkalStuZach2013}, and quantum groups of rapid decay \cite{BhowVoigtZach2015}.
	
		
		\begin{theorem}[Ozawa--Rieffel {\cite[Main Theorem 1.2]{OzRi2005hyperbolic}}] \label{thm: OzRi main theorem}
			Let $A$ be a unital $\mathrm{C}^*$-algebra and $(A_n)_{n \in \N}$ a filtration as above, and let $D = \sum_{n \in \N} n P_n$.  Let $\varphi$ be a state on $A$, and suppose that there is some constant $C$ such that for all $n, r, s \in \mathbb{N}$, we have
			\[
			\norm{P_r a P_s} \leq C \norm{a_n}_{\varphi} \text{ for all } a \in A_n \ominus A_{n-1}.
			\]
			Then $(A,B(H_\varphi), D)$ with representation $\pi_{\varphi}$ is a spectral metric space.
		\end{theorem}
		
		We are now ready to extend Ozawa and Rieffel's results to the definition and study of spectral metric spaces for our foundational example: the $q$-Gaussians.
		
		\subsection{Background on $q$-Gaussians} \label{subsec: q Gaussian setup}
		
		Let $H$ be a Hilbert space, and let $q \in [-1,1]$.  For $\sigma \in S_n$, let $\inv(\sigma)$ denote the number of inversions of $\sigma$, or
		\[
		\inv(\sigma) = |\{(i,j): i < j, \sigma(i) > \sigma(j) \}|.
		\]
		Define a sesquilinear form $\ip{\cdot,\cdot}_q$ on $H^{\otimes n}$ by
		\[
		\ip{\xi_1 \otimes \cdots \otimes \xi_n, \eta_1 \otimes \dots \otimes \eta_n}_q = \sum_{\sigma \in S_n} q^{\inv(\sigma)} \ip{\xi_1,\eta_{\sigma^{-1}(1)}} \dots \ip{\xi_n,\eta_{\sigma^{-1}(n)}}.
		\]
		Note that $\ip{\cdot,\cdot}_0$ is the standard inner product on $H^{\otimes n}$.  Define a unitary $U_\sigma: H^{\otimes n} \to H^{\otimes n}$ by
		\[
		U_\sigma[\xi_1 \otimes \dots \otimes \xi_n] = \xi_{\sigma^{-1}(1)} \otimes \dots \otimes \xi_{\sigma^{-1}(n)}.
		\]
		Let
		\[
		W_q^{(n)} = \sum_{\sigma \in S_n} q^{\inv(\sigma)} U_\sigma.
		\]
		Then for $\xi, \eta \in H^{\otimes n}$,
		\[
		\ip{\xi, \eta}_q = \ip{\xi, W_q^{(n)} \eta}_0.
		\]
		Lemma 4 of \cite{BoSp1991example} demonstrates that $W_q^{(n)}$ is a bounded operator with respect to $\ip{\cdot,\cdot}_0$ with norm less than or equal to 
		\[
		\prod_{i=0}^{n-1}\, (1+q+ \cdots q^i).
		\]  
		A nontrivial fact, also proved in \cite{BoSp1991example}, is that $W_q^{(n)}$ is positive semidefinite for $q \in [-1,1]$ and positive definite with bounded inverse if $q \in (-1,1)$.  Hence, $\ip{\cdot,\cdot}_q$ defines an inner product on $H^{\otimes n}$ for $q \in (-1,1)$.  In the case where $q = 1$ or $-1$, then it defines an inner product on a quotient Hilbert space, which is the symmetric tensor product when $q = 1$ and the antisymmetric tensor product when $q = 1$.
		
		\begin{definition}[{\cite{BoKuSp1997qgaussian}}] ~
			\begin{enumerate}[(1)]
				\item Let $(H^{\otimes n})_q$ be $H^{\otimes n}$ equipped with the inner product $\ip{\cdot,\cdot}_q$.  Then the \emph{$q$-Gaussian Fock space} is the Hilbert space
				\[
				\mathcal{F}_q(H) = \C \Omega \oplus \bigoplus_{n \geq 1} (H^{\otimes n})_q.
				\]
				Here $\Omega$ is a unit vector called the \emph{vacuum vector}.  By convention, the $(H^{\otimes 0})_q := \C \Omega$. 
				$\newline$
				\item Given $\xi \in H$, the \emph{creation operator} $c(\xi)$ and the \emph{annihilation operator} $a(\xi)$ on $\mathcal{F}_q(H)$ is given by 
				\begin{align*}
					c(\xi) \Omega & = \xi, \\ 
					c(\xi) \xi_1 \otimes \cdots \otimes \xi_n & = \xi \otimes \xi_1 \otimes \cdots \otimes \xi_n,
				\end{align*}
				and 
				\begin{align*}
					a(\xi) \Omega & = 0, \\ 
					a(\xi) \xi_1 \otimes \cdots \otimes \xi_n & = \sum_{i=1}^n q^{i-1} \ip{\xi, \xi_i}_q \xi \otimes \cdots \otimes \widecheck{\xi}_i \otimes \cdots \otimes \xi_n,
				\end{align*}
				where $\widecheck{\xi}_i$ signifies that $\xi_i$ is deleted in the tensor.
			\end{enumerate}
		\end{definition}
		
		With respect to $\ip{ \cdot, \cdot}_q$, Bo{\.z}ejko and Speicher demonstrated that the annihilation and creation operators are adjoints of each other (Lemma 2, \cite{BoSp1991example}).  In fact, they showed in Lemma 4 of \cite{BoSp1991example} that the operators $c(\xi)$ and $a(\xi)$ are bounded operators on $\mathcal{F}_q(H)$ with
		\[
		\norm{c(\xi)}_q = \norm{a(\xi)}_q = \begin{cases}
			\frac{\norm{\xi}}{\sqrt{1-q}}, & 0 \leq q <1, \\
			\norm{\xi}, & -1 < q \leq 0.
		\end{cases}
		\]
		Hence, in the rest of the paper, we will write $a(\xi)^*$ rather than $c(\xi)$ in keeping with the notation of \cite{BoSp1991example}.  The creation and annihilation operators satisfy the $q$-relations $$a(\xi)a^{\ast}(\eta) - q a^{\ast}(\eta)a(\xi) = \ip{\xi, \eta}_H \mathbf{1} \qquad \qquad \qquad (\xi, \eta \in H)$$
		by \cite[Lemma 1]{BoSp1991example}.  The $q$-Gaussian variables are the self-adjoint operators $a(\xi) + a(\xi)^*$, which serve as a generalization of the classical Gaussians that arise in the $q = 1$ case.
		
		\begin{definition}
			[\cite{BoKuSp1997qgaussian}]
			Let $H_{\R}$ be a real Hilbert space and  $H_{\C}= \C \otimes_{\R} H_{\R} = H \oplus iH$ its complexification.  For $\xi \in H_{\C}$, set
			$$X(\xi) := a(\xi) + a^{\ast}(\xi) \in B(\mathcal{F}_q(H_{\C})).$$  Denote by $\mathcal{A}_q(H_{\R}) \subset B(\mathcal{F}_q(H_{\C}))$ the $\Cst$-algebra generated by all $X(\xi)$, $$\mathcal{A}_q(H_{ \R}) := \Cst(\, X(\xi) : \xi \in H_{\R} \, ),$$ and by $E$ the vacuum expectation state on $\mathcal{A}_q(H_{\mathbb{R}})$ given by $$E(T) = \langle \Omega, T\Omega  \rangle_q.$$ 
		\end{definition}

		\subsection{Proof of Theorem \ref{thm: q Gaussian length}} \label{subsec: q Gaussian length proof}

		
		We can now prove that the length-like spectral triple defines a compact quantum metric space structure for $q$-Gaussians.
		
		\begin{proof}[Proof of Theorem \ref{thm: q Gaussian length}]
			The GNS Hilbert space of $\mathcal{A}_q(H_{\R})$ for the vacuum expectation state is naturally isomorphic to the original $\mathcal{F}_q(H)$ via the map $\widehat{x} \mapsto x \Omega$.  Note that for $\xi \in H_{\R}$,
			\[
			X(\xi) \Omega = a^*(\xi)\Omega + a(\xi)\Omega = \xi.
			\] 
			For more general simple tensors, consider $\xi_1 \otimes \cdots \otimes \xi_n$, where $n \in \N$ and $\xi_1, \cdots, \xi_n \in H_{\R}$.  The normal ordered representation for the Wick product gives $\Psi(\xi_1 \otimes \cdots \otimes \xi_n) = \widehat{x} \in \mathcal{A}_q(H_{\R})$ such that $\widehat{x} \Omega  = \xi_1 \otimes \cdots \otimes \xi_n$ \cite[Proposition 2.7]{BoKuSp1997qgaussian}, hence
			\[
			E(\widehat{x}^*\widehat{x}) = \langle \widehat{x}^*\widehat{x} \Omega, \Omega \rangle_q =   \langle \widehat{x}\Omega, \widehat{x}\Omega \rangle_q = \langle \xi_1 \otimes \cdots \otimes \xi_n, \xi_1 \otimes \cdots \otimes \xi_n \rangle_q
			\]
			which vanishes if and only if $\xi_1 \otimes \cdots \otimes \xi_n=0$.  To apply the Ozawa-Rieffel criterion Theorem \ref{thm: OzRi main theorem}, define the filtration 
			\[
			A_n = \Span((a(\xi_1) + a(\xi_1)^*) \dots (a(\xi_n) + a(\xi_n)^*): \xi_j \in H_{\R} \text{ for } j \leq n ) 
			\]
			Fix choices of $n, r, s \in \mathbb{N}$ and an orthonormal basis $(e_i)_{i \in I}$ for $H_{\R}$.  Let $\vec{\imath}= (i_1, \cdots, i_n)$ denote a vector with entries in $I$ and $\psi_{\vec{\imath}}$ the Wick product of $e_{i_1}\otimes \cdots \otimes e_{i_n}$.  Then each $f \in A_n$ can be written as
			\begin{align*}
				f = \sum_{\vec{\imath} \in I^n} \alpha_{\vec{\imath}} \psi_{\vec{\imath}} & = \sum_{\vec{\imath} \in I^n} \alpha_{\vec{\imath}} \, \sum_{k=0}^n \,\,  \sum_{\sigma \in S_n \setminus (S_{n-k}\times S_k)} q^{|\sigma|} a^*(e_{\sigma^{-1}(i_1)}) \cdots a^*(e_{\sigma^{-1}(i_{n-k})}) \, a(e_{\sigma^{-1}(i_{n-k+1})}) \cdots a(e_{\sigma^{-1}(i_n)}) \\
				&:= \sum_{\vec{\imath} \in I^n} \alpha_{\vec{\imath}} \, \sum_{k=0}^n \psi_{\vec{\imath}}^{(k)} \\
				&:= \sum_{k=0}^n f^{(k)},
			\end{align*}
			where $\alpha_{\vec{\imath}} \in \mathbb{C}$, $\psi_{\vec{\imath}}^{(k)}$ is a linear combination of operators which perform $k$ annihilations and $n-k$ creations, and $f^{(k)}$ is a linear combination of all such $\psi_{\vec{\imath}}^{(k)}$ for a given $k$.  In particular, 
			\[
			P_r \, f \, P_s = \sum_{k=0}^n P_r \, f^{(k)} \, P_s,
			\]
			where each $P_r \, f^{(k)} \, P_s$ takes an element in $(H^{\otimes s})_q$ to one in $(H^{\otimes s+n -2k})_q$.  Observe that  $P_r \, f^{(k)} \, P_s$ is trivial if $k > s$ or $r \neq s+n-2k$.  When $P_r \, f^{(k)} \, P_s$ is nontrivial, the latter condition on the indices uniquely determines $k$, in which case,
			\[
			\| P_r \, f \, P_s \| = \| P_r \, f^{(k)} \, P_s \| \leq  C_q \| f \|_2,
			\]
			where
			\[
			C_q^{-1} = \Pi_{m=1}^{\infty}(1 - q^m), 
			\]
			and the inequality follows from \cite[Proposition 2.1]{Bozejko1999ultracontractivity}.  Therefore, we have $\norm{P_r f P_s} \leq C_q \norm{f}_2$ whenever $f \in A_n$.  As the same bound holds in the trivial case when $P_r f P_s = 0$, operators like $P_n$ can be used to define a Dirac operator on $\mathcal{A}_q(H_{\mathbb{R}})$.
			
		\end{proof}
		
		\subsection{The fermionic case} \label{subsec: fermionic}
		
		In this section, we describe the spectral triple from Theorem \ref{thm: q Gaussian length} more explicitly in the fermionic ($q = -1$) case.  We use the realization of $\mathcal{A}_{-1}(\C^d)$, also known as the Clifford algebra, on $(\C^2)^{\otimes d}$ described in \cite[\S 2.2]{CarlenMaas2014analog} and \cite[\S 6]{CarlenMaas2017gradient}.  The number operator $D = D_{-1}$ used in Theorem \ref{thm: q Gaussian length} is well-studied in the fermionic setting and also has a Lindbladian description in terms of a sum of second-order commutators; see \cite[\S 6]{CarlenMaas2017gradient}.  We will describe the action of the commutator with $D$ on the elements of $\mathcal{A}_{-1}(\R^d)$ using the coordinates in $(\C^2)^{\otimes d}$.
		
		Consider $\C^2$ with standard basis $(e_0,e_1)$ and $\C^d$ with standard basis $(e_1,\dots,e_n)$.  Recall that $\mathcal{F}_{-1}(\R^d)$ is the antisymmetric tensor algebra generated by $\C^d$.  Thus, each simple tensor $e_{i(1)} \otimes \dots \otimes e_{i(k)}$ is identified with $\sgn(\sigma)(e_{i(\sigma(1))} \otimes \dots \otimes e_{i(\sigma(k))}$ for a permutation $\sigma$ of $[k]$.  In particular, the vector is zero if the indices $i(1)$, \dots, $i(k)$ are not distinct, and so the tensors of order larger than $d$ all vanish as well.  An orthonormal basis for $\mathcal{F}_{-1}(\R^d)$ can therefore by given by vectors of the form $e_{i(1)} \otimes \dots \otimes e_{i(k)}$ where $0 \leq k \leq d$ and $i(1) < i(2) < \dots < i(k)$ (with the $k = 0$ case being the vacuum vector).
		
		Let $V: \mathcal{F}_{-1}(\C^d) \to (\C^2)^{\otimes d}$ be the unitary transformation given as follows:  For each $k \in \{0,\dots,d\}$ and increasing function $i: [k] \to [d]$, let $\widehat{\imath}: [d] \to \{0,1\}$ be the indicator function of $i([k]) \subseteq [n]$, and set
		\[
		V(e_{i(1)} \otimes \dots \otimes e_{i(k)}) = e_{\widehat{\imath}(1)} \otimes \dots \otimes e_{\widehat{\imath}(d)}.
		\]
		In other words, $V$ sends the vector $e_{i(1)} \otimes \dots \otimes e_{i(d)}$ to the simple tensor in $(\C^2)^{\otimes d}$ which has $e_1$ in the position $i(1)$, \dots, $i(k)$ and $e_0$ in the other positions.
		
		Let $E_{i,j}$ for $i, j \in \{0,1\}$ denote the matrix units in $M_2(\mathbb{C})$, and recall the Pauli matrices
		\[
		\sigma_x = \begin{bmatrix} 0 & 1 \\ 1 & 0 \end{bmatrix}, \quad
		\sigma_y = \begin{bmatrix} 0 & i \\ -i & 0 \end{bmatrix}, \quad
		\sigma_z = \begin{bmatrix} 1 & 0 \\ 0 & -1 \end{bmatrix}.
		\]
		
		\begin{proposition}
			With the notation above, we have the following identities:
			\begin{enumerate}
				\item $V a(e_j)^* V^* = \sigma_z^{\otimes (j-1)} \otimes E_{1,0} \otimes I^{\otimes (n-j)}$.
				\item $V x_j V^* = \sigma_z^{\otimes (j-1)} \otimes \sigma_x \otimes I^{\otimes (d-j)}$.
				\item $V D V^* = \sum_{j=1}^n I^{\otimes (j-1)} \otimes E_{1,1} \otimes I^{\otimes (d-j)}$.
				\item $i V [D,x_j] V^* = \sigma_z^{\otimes (j-1)} \otimes \sigma_y \otimes I^{\otimes (d-j)}$.
			\end{enumerate}
		\end{proposition}
		
		\begin{proof}
			(1) Suppose that $i(1) < \dots < i(k)$ and consider the vector $e_{i(1)} \otimes \dots \otimes e_{i(k)} \in \mathcal{F}_{-1}(\C^d)$.
			
			\textbf{Case 1:} Suppose $j \in i([k])$.  Then $a(e_j)^*[e_{i(1)} \otimes \dots \otimes e_{i(k)}] = 0$ by antisymmetry of the tensor product.  Meanwhile, the vector $e_{\widehat{i}(1)} \otimes \dots \otimes e_{\widehat{i}(n)}$ has $e_1$ in the $j$th position, which will be annihilated by the operator $E_{1,0}$.
			
			\textbf{Case 2:} Suppose $j \not \in i([k])$, and let $m$ be the number of indices in $i([k])$ that are less than $j$.  Then by antisymmetry
			\[
			a(e_j)^*[e_{i(1)} \otimes \dots \otimes e_{i(k)}] = (-1)^m e_{i(1)} \otimes \dots \otimes e_{i(m)} \otimes e_j \otimes e_{i(m+1)} \otimes \dots \otimes e_{i(k)}.
			\]
			Meanwhile, $e_{\widehat{i}(1)} \otimes \dots \otimes e_{\widehat{i}(n)}$ has exactly $m$ occurrences of $e_1$ in the indices $\{1,\dots,j-1\}$ and the $j$th tensorand is $e_0$.  Therefore,
			\begin{align*}
				(\sigma_z^{\otimes (j-1)} \otimes E_{1,0} \otimes I^{\otimes (d-j)})(e_{\widehat{i}(1)} \otimes \dots \otimes e_{\widehat{i}(d)}) &= (-1)^m e_{\widehat{i}(1)} \otimes \dots \otimes e_{\widehat{i}(j-1)} \otimes e_1 \otimes e_{\widehat{i}(j+1)} \otimes \dots \otimes e_{\widehat{i}(d)} \\
				&= V[(-1)^m e_{i(1)} \otimes \dots \otimes e_{i(m)} \otimes e_j \otimes e_{i(m+1)} \otimes \dots \otimes e_{i(k)}].
			\end{align*}
			
			(2) This follows by adding (1) and its adjoint.
			
			(3) Recall that the operator $D$ multiplies the simple tensor $e_{i(1)} \otimes \dots \otimes e_{i(k)}$ by a factor of $k$.  The corresponding vector in $e_{\widehat{i}(1)} \otimes \dots \otimes e_{\widehat{i}(d)}$ will have exactly $k$ tensorands that are $e_1$ and $d-k$ that are $e_0$.  Hence, there are $k$ values of $j$ for which the operator $I^{\otimes (j-1)} \otimes E_{1,1} \otimes I^{\otimes (d-j)}$ will fix the vector and $d - k$ for which it annihilates the vector.  Hence, $\sum_{j=1}^n I^{\otimes (j-1)} \otimes E_{1,1} \otimes I^{\otimes (n-j)}$ will multiply the vector $e_{\widehat{i}(1)} \otimes \dots \otimes e_{\widehat{i}(d)}$ by a factor of $k$.
			
			(4) We already know that $[D,x_j] = a(e_j)^* - a(e_j)$ and so (4) follows from (1) and its adjoint.  Alternatively, (4) follows from (1) and (3) since $E_{1,1}$ commutes with $\sigma_z$ and $I$.
		\end{proof}
		
		\section{Lipschitz seminorms via non-commutative couplings} \label{sec: couplings}
		
		\subsection{Definition and duality} \label{subsec: universal Lipschitz norm}
		
		We will define Lipschitz seminorms for a $\mathrm{C}^*$-algebra $A$ in terms of $*$-homomorphisms $\alpha$ and $\beta$ from $A$ into some other $\mathrm{C}^*$-algebra $\tilde{A}$.  As motivation, suppose that $K$ is a compact subset of $[-1,1]^m$ equipped with the $\ell^\infty$ metric, and let $x_1$, \dots, $x_m$ be the coordinate functions.  Then we could take $\tilde{A} = \C$ and $\alpha(f) = f(a)$ and $\beta(f) = f(b)$ for two points $a$, $b \in K$.  Then
		\[
		\frac{|f(a) - f(b)|}{\max_j |a_j - b_j|} = \frac{|\alpha(f) - \beta(f)|}{\max_j |\alpha(x_j) - \beta(x_j)|},
		\]
		and thus $\norm{f}_{\Lip}$ is the supremum of this quantity over the pairs of $*$-homomorphisms $\alpha$, $\beta$ associated to points $a$ and $b$.  More generally, if $\tilde{A} = C(K')$ for some other compact space $K'$ and if $a, b: K' \to K$ are continuous, there are $*$-homomorphisms $\alpha, \beta: C(K) \to C(K')$ given by $\alpha(f) = f \circ a$ and $\beta(f) = f \circ b$, and
		\[
		\norm{f \circ a - f \circ b} \leq \norm{f}_{\operatorname{Lip}} \max_j \norm{x_j \circ a - x_j \circ b}.
		\]
		This motivates us to define a Lipschitz seminorm on a general $\mathrm{C}^*$-algebra $A$ with generating set $x_1$, \dots, $x_m$ by the formula
		\[
		\norm{f}_{\Lip} = \sup_{\substack{(\tilde{A},\alpha,\beta) \\ \alpha \neq \beta}} \frac{\norm{\alpha(f) - \beta(f)}}{\max_j \norm{\alpha(x_j) - \beta(x_j)}},
		\]
		where $(\tilde{A},\alpha,\beta)$ ranges over triples where $\tilde{A}$ is a $\mathrm{C}^*$-algebra and $\alpha, \beta: A \to \tilde{A}$ are $*$-homomorphisms such that $\alpha \neq \beta$.
		
		However, there is some question, or flexibility, over whether we want to allow all possible separable $\mathrm{C}^*$-algebras $\tilde{A}$ and embeddings $\alpha$, $\beta$.  For instance, if $A$ is commutative, it is natural to restrict to $*$-homomorphisms from $A$ into $\C$ or into some commutative $\mathrm{C}^*$-algebra.  It is also somewhat natural in a free probability setting to restrict the $*$-homomorphisms to those that preserve chosen traces on $A$ and $\tilde{A}$.  On the other hand, the definition here seems to be of interest for general $\mathrm{C}^*$-algebras, so it may also be natural to use \emph{all} $*$-homomorphisms into any separable $\mathcal{A}$ (considered up to isomorphism). Hence, we will allow our definition be a fixed set $\mathcal{C}$ of triples $(\tilde{A},\alpha,\beta)$.
		
		We also want to allow an infinite generating set, and so we will add the restriction that $\sup_{i \in I} \norm{\alpha(x_i) - \beta(x_i)}$ will always be finite.  Note that if we start with an arbitrary generating set, we can always rescale each element to be in the unit ball to fulfill this condition.
		
		\begin{definition} \label{def: Lipschitz seminorm via couplings}
			Let $A$ be a $\mathrm{C}^*$-algebra generated by $x = (x_i)_{i \in I}$ such that $\norm{x_i} \leq 1$.  Fix a set $\mathcal{C}$ of triples $(\tilde{A},\alpha,\beta)$ such that $\tilde{A}$ is a $\mathrm{C}^*$-algebra and $\alpha, \beta: A \to \tilde{A}$ are $*$-homomorphisms, and assume that $\mathcal{C}$ contains some element where $\alpha \neq \beta$.  For $f \in A$, let
			\[
			L_{x,\mathcal{C}}(f) = \sup_{\substack{(\tilde{A},\alpha,\beta) \in \mathcal{C} \\ \alpha \neq \beta}} \frac{\norm{\alpha(f) - \beta(f)}_{\tilde{A}}}{\sup_{i \in I} \norm{\alpha(x_i) - \beta(x_i)}_{\tilde{A}}}.
			\]
		\end{definition}
		
		\begin{example}[Subsets of Euclidean space]
			Suppose that $A = C(K)$ for some $K \subseteq \R^d$ and let $x_i$ be the $i$th coordinate function.  Let $\cC$ be the class of $*$-homomorphisms from $A$ to $\C$, which are always given by point evaluation.  Then $L_{x,\cC}(f)$ is the Lipschitz norm of $f$ with respect to the $\ell^\infty$ metric on $\R^d$.
		\end{example}
		
		\begin{example}[Spectral triples]
			Suppose that $(A,H,D)$ is a spectral triple with representation $\pi$.  Assume that $A$ is generated by $(x_i)_{i \in I}$ for some finite index set $I$.  Assume that $[D,\pi(x_i)]$ is bounded for each $i$ and $[D,\pi(x_i)]$ is nonzero for some $i$.  For $t \in \R$, let $\alpha_t: A \to B(H)$ be the map $\alpha_t(a) = e^{itD} \pi(a) e^{-itD}$; we will consider the class $\cC = \{\alpha_t: t \in [0,t_0]\}$ for suitable $t_0 > 0$, chosen below.    We claim that
			\begin{equation} \label{eq: Dirac conjugation estimate}
				\sup_{t \geq 0} t^{-1} \norm{\alpha_t(a) - \pi(a)}_{B(H)} \leq \norm{[\pi(a),D]}_{B(H)} = \lim_{t \to 0^+} t^{-1} \norm{\alpha_t(a) - \pi(a)}_{B(H)}.
			\end{equation}
			To see this, fix a sequence of projections $P_n \leq P_{n+1}$ in $B(H)$ which converge to $1$ in strong operator topology, such that $P_n D = D P_n$ is bounded for each $n$.  Then
			\begin{multline*}
				\frac{d}{dt} P_n \alpha_t(a) P_n = \frac{d}{dt} P_n e^{itD} \pi(a) e^{-itD} P_n = \frac{d}{dt} e^{itD} P_n \pi(a) P_n e^{-itD} \\
				= e^{itD} i[D, P_n\pi(a) P_n] e^{-itD} = P_n e^{itD} [D,\pi(a)] e^{-itD} P_n,
			\end{multline*}
			where the differentiation is justified since $D$ is bounded on $P_n(H)$.  As a consequence,
			\[
			\norm{P_n(\alpha_t(a) - \pi(a))P_n}_{B(H)} \leq t \norm{P_n[D,\pi(a)]P_n}_{B(H)} \leq t \norm{[D,\pi(a)]}_{B(H)}.
			\]
			Taking $n \to \infty$ on the left-hand side shows that $\norm{\alpha_t(a) - \pi(a)}_{B(H)} \leq t \norm{[D,\pi(a)]}_{B(H)}$, hence showing the first inequality of \eqref{eq: Dirac conjugation estimate}.  This of course also implies that
			\[
			\limsup_{t \to 0^+} t^{-1} \norm{\alpha_t(a) - \pi(a)}_{B(H)} \leq \norm{[D,\pi(a)]}_{B(H)},
			\]
			and on the other hand,
			\[
			\liminf_{t \to 0^+} t^{-1} \norm{\alpha_t(a) - \pi(a)}_{B(H)} \geq \liminf_{t \to 0^+} t^{-1} \norm{P_n(\alpha_t(a) - \pi(a))P_n}_{B(H)} = \norm{P_n[D,\pi(a)]P_n}_{B(H)}.
			\]
			Taking $n \to \infty$ on the right-hand side shows that $\liminf_{t \to 0^+} t^{-1} \norm{\alpha_t(a) - \pi(a)}_{B(H)} \geq \norm{[D,\pi(a)]}_{B(H)}$, completing the proof of \eqref{eq: Dirac conjugation estimate}.  Recall we assumed that $[D,\pi(x_i)] \neq 0$ for some $i$.  Thus, using \eqref{eq: Dirac conjugation estimate}, given $\varepsilon \in (0,1)$, there is some $t_0 > 0$ such that
			\[
			(1 - \varepsilon) \max_{i \in I} \norm{[D,\pi(x_i)]}_{B(H)} \leq t^{-1} \max_{i \in I} \norm{\alpha_t(x_i) - \pi(x_i)}_{B(H)} \leq \max_{i \in I} \norm{[D,\pi(x_i)]}_{B(H)} \text{ for all } t \in [0,t_0].
			\]
			Letting $\cC = \{\alpha_t: t \in [0,t_0]\}$ and $K = \max_{i \in I} \norm{[D,\pi(x_i)]}_{B(H)}$, and using the semigroup property for $e^{itD}$, we have for $a \in A$ and $s, t \in [0,t_0]$,
			\[
			(1 - \varepsilon) K |s - t| \leq \max_{i \in I} \norm{\alpha_s(x_i) - \alpha_t(x_i)}_{B(H)} \leq K |s - t|.
			\]
			We therefore have that
			\[
			\frac{\norm{\alpha_s(a) - \alpha_t(a)}_{B(H)}}{(1 - \varepsilon)K|s-t|} \leq \frac{\norm{\alpha_s(a) - \alpha_t(a)}_{B(H)}}{\max_{i \in I} \norm{\alpha_s(x_i) - \alpha_t(x_i)}_{B(H)}} \leq \frac{\norm{\alpha_s(a) - \alpha_t(a)}_{B(H)}}{K|s-t|}.
			\]
			Taking the supremum over $s, t \in [0,t_0]$, and applying \eqref{eq: Dirac conjugation estimate} and the semigroup property, we get
			\[
			\frac{1}{(1-\varepsilon)K} \norm{[D,\pi(a)]}_{B(H)} \leq L_{x,\cC}(a) \leq \frac{1}{K} \norm{[D,\pi(a)]}_{B(H)}.
			\]
			Thus, our construction recovers the spectral triple Lip-norm up to constants which can be taken arbitrarily close to $1/K$.
		\end{example}
		
		\begin{example}[Trace-preserving $*$-homomomorphisms] \label{ex: trace preserving}
			Let $A$ be a $\mathrm{C}^*$-algebra with a faithful trace $\tau$.   Let $\mathcal{C}_{\tr} = \mathcal{C}_{\tr}(A,\tau)$ be the class of $(\tilde{A},\alpha,\beta)$ where $\tilde{A}$ is a $\mathrm{C}^*$-algebra with a faithful trace $\tilde{\tau}$ and $\alpha, \beta: A \to \tilde{A}$ are unital $*$-homomorphisms with $\tilde{\tau} \circ \alpha = \tau$ and $\tilde{\tau} \circ \beta = \tau$.  We remark that these pairs are quite similar to the non-commutative couplings used by Biane and Voiculescu to define the free Wasserstein distance in \cite{BV2001}.
		\end{example}
		
		
		\begin{proposition}[Basic properties of $L_{x,\mathcal{C}}$] \label{prop: Lipschitz seminorm basic properties}
			Consider the setup of Definition \ref{def: Lipschitz seminorm via couplings}.  Let $f, g \in A$ and let $\varphi$ be a state on $A$.
			\begin{enumerate}[(1)]
				\item $\norm{f - \varphi(f)}_A \leq 2 L_{x,\mathcal{C}}(f) \sup_i \norm{x_i}_A$, provided that $\mathcal{C}$ contains $(A \otimes A, \iota_1,\iota_2)$, where $\otimes$ is the $\mathrm{C}^*$-minimal tensor product and $\iota_1(a) = a \otimes 1$ and $\iota_2(a) = 1 \otimes a$.  In particular, in this case, the kernel of $L_{x,\mathcal{C}}$ is $\mathbb{C} 1$. 
				\item $L_{x,\mathcal{C}}$ satisfies the Leibniz inequality $L_{x,\mathcal{C}}(fg) \leq L_{x,\mathcal{C}}(f) \norm{g}_A + \norm{f}_A L_{x,\mathcal{C}}(f)$.
				\item $L_{x,\mathcal{C}}(x_i) \leq 1$.
				\item The domain $\dom(L_{x,\mathcal{C}}) = \{f \in A: L_{x,\mathcal{C}}(f) < \infty\}$
				contains the $*$-algebra generated by $(x_i)_{i \in I}$, and in particular is norm-dense in $A$.
				\item $L_{x,\mathcal{C}}$ is lower semi-continuous on $A$.
			\end{enumerate}
		\end{proposition}
		
		\begin{proof}
			(1) Let $(A \otimes A,\iota_1,\iota_2)$ be as above.  Note that $\id \otimes \varphi$ defines a unital completely positive map $A \otimes A \to A$.  Then
			\begin{align*}
				\norm{f - \varphi(f)}_{A} &= \norm{(\id \otimes \varphi)(\iota_1(f) - \iota_2(f))}_A \\
				&\leq \norm{\iota_1(f) - \iota_2(f)}_{A \otimes A} \\
				&\leq L_{x,\mathcal{C}}(f) \sup_{i \in I} \norm{\iota_1(x_i) - \iota_2(x_i)}_{A \otimes A} \\
				&\leq L_{x,\mathcal{C}}(f) \cdot 2 \sup_{i \in I} \norm{x_i}_A
			\end{align*}
			
			(2) Note that for any $\tilde{A}$ and $*$-homomorphisms $\alpha, \beta: A \to \tilde{A}$, we have
			\[
			\norm{\alpha(fg) - \beta(fg)}_{\tilde{A}} = \norm{(\alpha(f) - \beta(f))\alpha(g) + \beta(f)(\alpha(g) - \beta(g))}_{\tilde{A}} \leq \norm{\alpha(f) - \beta(f)}_{\tilde{A}} \norm{g}_A + \norm{f}_A \norm{\alpha(g) - \beta(g)}_{\tilde{A}}.
			\]
			Dividing by $\sup_{i \in I} \norm{\alpha(x_i) - \beta(x_i)}_{\tilde{A}}$ implies the desired bound.
			
			(3) This is immediate because $\norm{\alpha(x_j) - \beta(x_j)}_{\tilde{A}} \leq \sup_{i \in I} \norm{\alpha(x_i) - \beta(x_i)}_{\tilde{A}}$.
			
			(4) This follows from (2) and (3).
			
			(5) For each fixed $(\tilde{A},\alpha,\beta) \in \mathcal{C}$ with $\alpha \neq \beta$,
			\[
			f \mapsto \frac{\norm{\alpha(f) - \beta(f)}_{\tilde{A}}}{\max_j \norm{\alpha(x_j) - \beta(x_j)}}_{\tilde{A}}
			\]
			is continuous on $A$.  Thus, $L_{x,\mathcal{C}}$ is a supremum of a family of continuous functions and hence is lower semi-continuous.
		\end{proof}
		
		\begin{remark} \label{rem: trace preserving finite diameter}
			If we assume that $A$ is equipped with a trace $\tau$ and consider the trace-preserving couplings as in \ref{ex: trace preserving}, then this will include the tensor inclusions $\iota_1$ and $\iota_2$ of $(A,\tau)$ into $(A \otimes A, \tau \otimes \tau)$, and so Proposition \ref{prop: Lipschitz seminorm basic properties} (1) applies in this case.
		\end{remark}
		
		\begin{proposition}[Monge--Kantorovich metric and duality] \label{prop: MK duality}
			Consider the setup of Definition \ref{def: Lipschitz seminorm via couplings} and assume that $\mathcal{C}$ contains $(A \otimes A,\iota_1,\iota_2)$. For $\varphi$, $\psi \in \mathcal{S}(A)$, let
			\[
			d_{x,\mathcal{C}}(\varphi,\psi) = \sup \{ |\varphi(f) - \psi(f)|: f \in A_{\sa}, L_{x,\mathcal{C}}(f) \leq 1\}.
			\]
			Then $d_{x,w,\mathcal{C}}$ is a metric on $\mathcal{S}(A)$.  for $f \in A_{\sa}$, we have the noncommutative Monge--Kantorovich duality
			\begin{equation} \label{eq: MK duality}
				L_{x,\mathcal{C}}(f) = \sup \left\{\frac{|\varphi(f) - \psi(f)|}{d_{x,\mathcal{C}}(\varphi,\psi)}: \varphi, \psi \in \mathcal{S}(A), \varphi \neq \psi \right\}.
			\end{equation}
		\end{proposition}
		
		\begin{remark}
			The duality \eqref{eq: MK duality} follows from \cite[Theorem 4.1]{Rieffel2} because of the lower-semicontinuity of $L_{x,\mathcal{C}}$ proved in Proposition \ref{prop: Lipschitz seminorm basic properties} (5), since $\ker(L_{x,\mathcal{C}}) = \C 1$ by Proposition \ref{prop: Lipschitz seminorm basic properties} (1).  However, we give an additional proof of \eqref{eq: MK duality} below directly from the definition of $L_{x,\mathcal{C}}$.
		\end{remark}
		
		\begin{proof}[Proof of Proposition \ref{prop: MK duality}]
			In light of Proposition \ref{prop: Lipschitz seminorm basic properties} (1), $d_{x,\mathcal{C}}(\varphi,\psi)$ is finite.  By Proposition \ref{prop: Lipschitz seminorm basic properties} (4), $\dom(L_{x,\mathcal{C}})$ is dense in $A$, which implies that if $\varphi \neq \psi$, then $d_{x,w}(\varphi,\psi) > 0$.  Nonnegativity, symmetry, and the triangle inequality are clear from the construction of $d_{x,\mathcal{C}}$.
			
			It remains to prove the Monge--Kantorovich duality.  The inequality ($\geq$) is immediate.  To prove the opposite inequality, assume without loss of generality that $L_{x,\mathcal{C}}(f) \neq 0$, and let $\varepsilon > 0$, and let $(\tilde{A},\alpha,\beta) \in \mathcal{C}$ such that
			\[
			\frac{\norm{\alpha(f) - \beta(f)}_{\tilde{A}}}{\sup_{i \in I} \norm{\alpha(x_i) - \beta(x_i)}_{\tilde{A}}} > (1 - \varepsilon) L_{x,\mathcal{C}}(f).
			\]
			Since $\alpha(f) - \beta(f)$ is self-adjoint, it generates a commutative $\mathrm{C}^*$-algebra $B$, and there is a state $\theta_0$ on $B$ with $|\theta_0(\alpha(f) - \beta(f))| = \norm{\alpha(f) - \beta(f)}$.  By the Hahn--Banach theorem, $\theta_0$ extends to a state $\theta$ on $\tilde{A}$.  Consider the states $\varphi = \theta \circ \alpha$ and $\psi = \theta \circ \beta$ on $A$.  By construction $\varphi - \psi = \theta \circ (\alpha - \beta)$, and hence for $g \in A$,
			\[
			|\varphi(g) - \psi(g)| \leq \norm{\theta} \norm{\alpha(g) - \beta(g)}_{\tilde{A}} \leq L_{x,\mathcal{C}}(g) \sup_{i \in I} \norm{\alpha(x_i) - \beta(x_i)}_{{\tilde{A}}}.
			\]
			Therefore,
			\[
			d_{x,\mathcal{C}}(\varphi,\psi) \leq \sup_{i \in I} \norm{\alpha(x_i) - \beta(x_i)}_{\tilde{A}}.
			\]
			Moreover,
			\begin{align*}
				|\varphi(f) - \psi(f)| &= |\theta(\alpha(f) - \beta(f))| \\
				&= \norm{\alpha(f) - \beta(f)}_{\tilde{A}} \\
				&\geq (1 - \varepsilon) L_{x,\mathcal{C}}(f) \sup_{i \in I} \norm{\alpha(x_i) - \beta(x_i)}_{\tilde{A}} \\
				&\geq (1 - \varepsilon) L_{x,\mathcal{C}}(f) d_{x,\mathcal{C}}(\varphi,\psi).
			\end{align*}
			Therefore,
			\[
			\sup \left\{\frac{|\varphi(f) - \psi(f)|}{d_{x,\mathcal{C}}(\varphi,\psi)}: \varphi, \psi \in \mathcal{S}(A), \varphi \neq \psi \right\} \geq (1 - \varepsilon) L_{x,\mathcal{C}}(f),
			\]
			and since $\varepsilon$ was arbitrary, the proof is complete.
		\end{proof}
		
		Since the Lipschitz norms defined here depend on the choice of generators, it is natural to consider what happens if we change the generating set.
		
		\begin{lemma}[Change of coordinates] \label{lem: change of coordinates}
			Let $A$ be a $\mathrm{C}^*$-algebra and consider two generating sets $x = (x_i)_{i \in I}$ and $y = (y_j)_{j \in J}$ in the unit ball. Let $\mathcal{C}$ be a class of triples $(\tilde{A},\alpha,\beta)$ as in Definition \ref{def: Lipschitz seminorm via couplings}.  Let $R \in (0,\infty)$.  Then the following are equivalent:
			\begin{enumerate}[(1)]
				\item $\sup_{j \in J} L_{x,\mathcal{C}}(y_j) \leq R$.
				\item $L_{x,\mathcal{C}}(f) \leq R \, L_{y,\mathcal{C}}(f)$ for all $f \in A$.
				\item $d_{y,\mathcal{C}}(\varphi,\psi) \leq R \, d_{x,\mathcal{C}}(\varphi,\psi)$ for $\varphi$, $\psi \in \mathcal{S}(A)$.
			\end{enumerate}
		\end{lemma}
		
		\begin{proof}
			Assume (1).  If $(\tilde{A},\alpha,\beta) \in \mathcal{C}$ and $f \in A$, then
			\begin{align*}
				\norm{\alpha(f) - \beta(f)}_{\tilde{A}} &\leq L_{y,\mathcal{C}}(f) \sup_{j \in J} \norm{\alpha(y_j) - \beta(y_j)} \\
				&\leq L_{y,\mathcal{C}}(f) \sup_{j \in J} L_{x,\mathcal{C}}(y_j) \sup_{i \in I} \norm{\alpha(x_i) - \beta(x_i)} \\
				&\leq R L_{y,\mathcal{C}}(f),
			\end{align*}
			hence (2) holds.
			
			Next, assume (2).  If $\varphi, \psi \in \mathcal{S}(A)$, then
			\[
			\sup_{f: L_{y,\mathcal{C}}(f) \leq 1} |\varphi(f) - \psi(f)| \leq \sup_{f: L_{x,\mathcal{C}}(f) \leq R} |\varphi(f) - \psi(f)| = R \sup_{f: L_{x,\mathcal{C}}(f) \leq 1} |\varphi(f) - \psi(f)|,
			\]
			hence (3) holds.
			
			Finally, assume (3).  By Proposition \ref{prop: MK duality},
			\begin{align*}
				L_{x,\mathcal{C}}(y_j) &= \sup \left\{\frac{|\varphi(y_j) - \psi(y_j)|}{d_{x,\mathcal{C}}(\varphi,\psi)}: \varphi, \psi \in \mathcal{S}(A), \varphi \neq \psi \right\} \\
				&\leq R \sup \left\{\frac{|\varphi(y_j) - \psi(y_j)|}{d_{y,\mathcal{C}}(\varphi,\psi)}: \varphi, \psi \in \mathcal{S}(A), \varphi \neq \psi \right\} \\
				&= R L_{y,\mathcal{C}}(y_j) \\
				&\leq R,
			\end{align*}
			where the last inequality follows from Proposition \ref{prop: Lipschitz seminorm basic properties}(3), hence (1) holds.
		\end{proof}
		
		\begin{corollary}
			In the situation of the previous lemma, $L_{x,\mathcal{C}}$ and $L_{y,\mathcal{C}}$ are equivalent up to constants (and $d_{y,\mathcal{C}}$ and $d_{x,\mathcal{C}}$ are equivalent up to constants) if and only if $\sup_{j \in J} L_{x,\mathcal{C}}(y_j) < \infty$ and $\sup_{i \in I} L_{y,\mathcal{C}}(x_i) < \infty$.
			
			Using Proposition \ref{prop: Lipschitz seminorm basic properties} (4), this always holds when $I$ and $J$ are finite sets and $(x_i)_{i \in I}$ and $(y_j)_{j \in J}$ generate the same $*$-algebra.
		\end{corollary}
		
		\subsection{Total boundness via semigroup regularization} \label{subsec: total boundedness}
		
		For $L_{x,w}$ to define a \emph{bona fide} compact quantum metric space, it remains to show that $\{f \in \mathrm{C}^*(x_i: i \in I): L_{x,\cC}(f) \leq 1, \tau(f) = 0\}$ is totally bounded in operator norm.  We do not expect this to be true in all cases.  The next proposition gives sufficient conditions in terms of an ultracontractive quantum Markov semigroup; it is designed to be applied to the $q$-Gaussians, but we state the hypotheses in a general form.
		
		\begin{proposition}
			\label{prop: total boundedness via semigroup}
			Let $(M,\tau)$ be a tracial von Neumann algebra generated by $(x_i)_{i \in I}$, and let $A = \mathrm{C}^*(x_i: i \in I)$, and assume that $R := \sup_i \norm{x_i} < \infty$.  Suppose that there is a quantum Markov semigroup $(\Phi_t)_{t \in [0,\infty)}$ satisfying the following properties:
			\begin{enumerate}[(1)] 
				\item \emph{Compactness:} $\Phi_t$ defines a compact operator on $L^2(M,\tau)$ for each $t > 0$.
				\item \emph{Ultracontractivity:} $\Phi_t$ gives a bounded operator $L^2(M,\tau) \to M$ for each $t > 0$.
				\item \emph{Continuous dilation:} There exists a tracial von Neumann algebra $(N,\tau_N)$, an embedding $\iota: (M,\tau_M) \to (N,\tau_N)$, and one-parameter group of automorphisms $\alpha_t$ such that
				\begin{itemize}
					\item $\Phi_t = \iota^* \alpha_t \iota$, where $\iota^*$ is the conditional expectation $N \to M$ adjoint to $\iota$.
					\item For each $t > 0$, the element $(N,\iota|_A,\alpha_t \circ \iota|_A)$ is in the class $\cC$ used to define the Lipschitz norm.
					\item $\lim_{t \to 0} \sup_i \norm{\alpha_t \circ \iota(x_i) - \iota(x_i)} = 0$.
				\end{itemize}
			\end{enumerate}
			Then $\{f \in M: L_{x,\cC}(f) \leq 1, \tau(f) = 0 \}$ is totally bounded in operator norm.  In particular, $(A,L_{x,\cC})$ is a Leibniz compact quantum metric space.
		\end{proposition}
		
		\begin{remark}
			It may be natural to assume that $\Phi_t$ maps $A$ into $A$ (and hence also $L^2(A,\tau)$ into $A$), as it seems difficult to imagine that (2) or (3) holds without $\Phi_t(A) \subseteq A$.  However, since the proof does not require $\Phi_t(A) \subseteq A$, we omit this assumption.
		\end{remark}
		
		\begin{proof}[{Proof of Proposition \ref{prop: total boundedness via semigroup}}]
			Let $S = \{f \in M: L_{x,w}(f) \leq 1, \tau(f) = 0 \}$, and recall that $S$ is contained in the norm-ball $B_R^M(0)$ by Proposition \ref{prop: Lipschitz seminorm basic properties} (1).  To show total boundedness of $S$, it suffices to show that for each $\varepsilon > 0$, $S$ is contained in the $\varepsilon$-neighborhood of some totally bounded set.  Note that for $f \in S$,
			\begin{align*}
				\norm{\Phi_t(f) - f}_M &= \norm{\iota^* \alpha_t \iota(f) - f}_M \\
				&= \norm{\iota^*(\alpha_t \iota(f) - \iota(f))}_M \\
				&\leq \norm{\alpha_t \iota(f) - \iota(f)}_N \\
				&\leq L_{x,\cC}(f) \sup_{i \in I} \norm{\alpha_t \iota(x_i) - \iota(x_i)}_N \\
				&\leq \sup_{i \in I} \norm{\alpha_t \iota(x_i) - \iota(x_i)}_N,
			\end{align*}
			where we have applied the first two conditions on the dilation and the definition of $L_{x,\cC}$.  Then by the third condition on the dilation, there exists some $t$ such that $\sup_{i \in I} \norm{\alpha_t \iota(x_i) - \iota(x_i)}_N < \varepsilon$.  Therefore, $S$ is contained in the $\varepsilon$-neighborhood of $\Phi_t(B_R^M(0))$.
			
			Hence, it suffices to show that $\Phi_t(B_R^M(0))$ is totally bounded.  Since $B_R^M(0)$ is a bounded subset of $L^2(M,\tau)$ and we assumed (1) that $\Phi_{t/2}$ defines a compact operator on $L^2(M,\tau)$, we see that $\Phi_{t/2}(B_R^M(0))$ is totally bounded in $L^2(M,\tau)$.  Then by the ultracontracitivity assumption (2), $\Phi_{t/2}$ is a bounded map from $L^2(M,\tau)$ to $M$, and so it maps totally bounded subsets of $L^2(M,\tau)$ to totally bounded subsets of $M$.  Hence, $\Phi_t(B_R^M(0)) = \Phi_{t/2}(\Phi_{t/2}(B_R^M(0))$ is totally bounded in $M$, as desired.
			
			The total boundedness of $S$, together with Proposition \ref{prop: Lipschitz seminorm basic properties}, shows that $(A,L_{x,\cC})$ is a compact quantum metric space.
		\end{proof}
		
		\begin{corollary} \label{cor: general Lip norm for q Gaussian}
			Consider the $q$-Gaussian algebra $\mathcal{A}_q(\R^d)$ with its canonical generators $(x_i)_{i=1}^d$.  Let $\mathcal{C} = \mathcal{C}_{\tr}(\mathcal{A}_q(\R^d),\tau)$ be the class of trace-preserving embeddings as in Example \ref{ex: trace preserving}.  Then $L_{x,\mathcal{C}}$ defines a compact quantum metric space structure on $\mathcal{A}_q(\R^d)$.
		\end{corollary}
		
		\begin{proof}
			We apply Proposition \ref{prop: total boundedness via semigroup} taking $\Phi_t$ to be the $q$-Gaussian Ornstein--Uhlenbeck semigroup as in \cite{Bozejko1999ultracontractivity}.  We recall that the action of $\Phi_t$ on $L^2(\mathcal{A}_q(\R^d),\tau)$ is given in the Fock space model by $e^{-tD_q}$ where $D_q$ is the length operator we used in \S \ref{subsec: q Gaussian length proof}, which is clearly a compact operator.  The ultracontractivity of $\Phi_t$ is Bo{\.z}ejko's main result in \cite{Bozejko1999ultracontractivity}.
			
			Next, we recall the dilation of $\Phi_t$ to a family of automorphisms via Bo{\.z}ejko, K{\"u}mmerer, and Speicher's $q$-Gaussian functor.  Let $H_{\R} = L^2(\R;\R) \otimes_{\R} \R^d$.  Define an automorphism $\hat{\alpha}_t$ of $L^2(\R) \otimes_{\R} \R^d = L^2(\R;\R^d)$ by $\hat{\alpha}_t(f)(s) = f(s - 2t)$.  Let
			\[
			f_0(t) = \begin{cases}
				e^{t/2}, & t \leq 0 \\
				0, & t > 0.
			\end{cases},
			\]
			so that $f_0$ is a unit vector in $L^2(\R;\R)$.  Define an embedding $\hat{\iota}: \R^d \to H_{\R}$ by $I(e_j) = f_0 \otimes e_j$.  We note that for $t \geq 0$
			\[
			\ip{f_0 \otimes v, \widehat{\alpha}_t f_0 \otimes w} = \int_{-\infty}^0 e^{s/2} e^{(s-2t)/2}\,ds \, \ip{v,w} = e^{-t} \ip{v,w}.
			\]
			Therefore,
			\[
			\hat{\iota}^* \hat{\alpha}_t \hat{\iota}(v) = e^{-t} v.
			\]
			By \cite[Theorem 2.11]{BoKuSp1997qgaussian}, $\hat{\iota}$ induces a trace-preserving inclusion $\iota: \mathcal{A}_q(\R^d) \to \mathcal{A}_q(H_{\R})$, $\hat{\iota}^*$ induces the corresponding conditional expectation $\mathcal{A}_q(H_{\R}) \to \mathcal{A}_q(\R^d)$, $\hat{\alpha}_t$ induces an automorphism of $\mathcal{A}_q(H_{\R})$.  By functoriality, the $\iota^* \alpha \iota$ is the map on $\mathcal{A}_q(\R^d)$ induced by the operator $e^{-t}1$ on $\R^d$, which is exactly the semigroup $\Phi_t$ (see \cite[Example 4.11]{BoKuSp1997qgaussian}).
			
			This dilation satisfies all the properties we need in Proposition \ref{prop: total boundedness via semigroup}.  Indeed, the maps $\alpha_t \circ \iota$ are trace-preserving and hence $(\mathcal{A}_q(H_{\R}),\iota,\alpha_t \circ \iota)$ is in the class $\mathcal{C}$.  Moreover, $\alpha_t \circ \iota(x_i) - \iota(x_i)$ is the $q$-Gaussian operator associated to the vector $(\hat{\alpha}_t - \id)(f_0 \otimes e_i)$.  As $t \to 0$, the norm of the vector vanishes, and hence so does the norm of the corresponding $q$-Gaussian operator.  Thus, $\lim_{t \to 0^+} \max_i \norm{\alpha_t \circ \iota(x_i) - \iota(x_i)} = 0$.  Therefore, Proposition \ref{prop: total boundedness via semigroup} implies the desired conclusions.
		\end{proof}
		
		\subsection{Application to free Gibbs laws} \label{subsec: free transport}
		
		Free Gibbs laws are the analog of measures on $\R^d$ with density $e^{-V}$.  Taking a self-adjoint non-commutative polynomial $f$ such that $e^{-n^2 \tr(f)}$ is integrable on $(\mathbb{M}_n)_{\sa}^d$, one can then set $V = \tr(f)$ and define a probability measure $\mu_{V}^{(n)}$ on $(\mathbb{M}_n)_{\sa}^d$ by
		\begin{equation} \label{eq: matrix model from potential}
			d\mu_V^{(n)}(X) = \frac{e^{-n^2 V(X)}}{\int_{(\mathbb{M}_n)_{\sa}^d} e^{-n^2 V}}\,dX.
		\end{equation}
		Letting $X^{(n)}$ be a random variable with distribution $\mu_V^{(n)}$, one can ask whether $\lim_{n \to \infty} \mathbb{E} \tr(p(X^{(n)}))$ exists for all non-commutative polynomials $p$.  This holds in various settings where $V$ is strongly convex \cite{GMS2006,GS2009,Jekel2020Entropy,JekelLiShlyakhtenko2022} (with some additional analytic hypotheses that are slightly different in each paper). In fact, one can then show that there exists some tuple $x$ in a tracial von Neumann algebra $(M,\tau)$ such that
		\[
		\lim_{n \to \infty} \tr(p(X^{(n)})) = \tau(p(x)) \text{ almost surely}
		\]
		for every $p$.  Our goal in this section is to describe a quantum metric space structure on the $\mathrm{C}^*$-algebra $A$ generated by $x$ using trace-preserving embeddings as in Example \ref{ex: trace preserving}.
		
		We will state and prove the result in the general setting of non-commutative smooth functions from \cite{JekelLiShlyakhtenko2022}.  To this end, we first recall the definition of non-commutative smooth functions.  They are constructed as a completion of the space of trace polynomials with respect to certain seminorms.  We first recall that $d$-variable \emph{trace polynomials} are the algebra spanned by formal terms of the form $p_0 \tr(p_1) \dots \tr(p_m)$ where $p_0$, \dots, $p_m$ are $d$-variable non-commutative polynomials and $\tr(p_j)$ is considered here as a formal symbol, under the following natural relations:
		\begin{itemize}
			\item $\tr(p)$ is linear in $p$.
			\item $\tr(pq) = \tr(qp)$
			\item $\tr(p)q = q \tr(p)$ and $\tr(p) \tr(q) = \tr(p) \tr(q)$.
			\item $\tr(1) = 1$
			\item The multiplication is given by
			\[
			(p_0 \tr(p_1) \dots \tr(p_m))(q_0 \tr(q_1) \dots \tr(q_n)) = (p_0 q_0) \tr(p_1) \dots \tr(p_m) \tr(q_1) \dots \tr(q_n).
			\]
		\end{itemize}
		For a more precise explanation of the definition, see \cite[\S 3.1]{JekelLiShlyakhtenko2022}.
		
		Given a $\mathrm{C}^*$-algebra $A$ with a trace $\tau$ and $x \in A_{\sa}^d$, one can define the evaluation $\ev_{A,\tau,x}(f)$ for trace polynomials $f$ as the unique algebra homomorphism such that
		\[
		\ev_{A,\tau,x}(\tr(p)) = \tau(p(x)), \qquad \ev_{A,\tau,x}(p) = p(x),
		\]
		for non-commutative polynomials $p$.  Let $\ev_{A,\tau}(f): A_{\sa}^d \to A$ be given by $\ev_{A,\tau}(f)(x) = \ev_{A,\tau,x}(f)$.  It turns out that $\ev_{A,\tau}(f)$ is a Fr{\'e}chet smooth function $A_{\sa}^d \to A$.  Let $\mathcal{M}_d^k(A)$ denote the set of real $k$-multilinear maps $(A^d)^k \to A$, and let
		\[
		\partial^k \ev_{A,\tau}(f): A_{\sa}^d \to \mathcal{M}_d^k(A)
		\]
		by the $k$th Fr{\'e}chet derivative of $\ev_{A,\tau,x}(f)$.  We then evaluate certain norms of the multilinear map $\partial^k \ev_{A,\tau,x}(f)$ given as follows.  Note that every real $k$-multilinear map $\Lambda: (A_{\sa}^d)^k \to A$ extends uniquely to a complex $k$-multilinear map $\Lambda_{\C}: (A^d)^k \to A$.  We furthermore define
		\begin{align*}
			\norm{\Lambda}_{\mathcal{M}_d^k(A)} &= \sup \{ \norm{\Lambda_{\C}(y_1,\dots,y_k)}_r: \\
			& \qquad y_j \in A^d, \norm{y_j}_{r_j} \leq 1,\\
			&\qquad r_j \in [1,\infty], r_1^{-1} + \dots + r_k^{-1} = r^{-1} \},
		\end{align*}
		where $\norm{y}_r$ is the non-commutative $L^r$-norm defined for $y \in A^d$ by
		\[
		\norm{y}_r := \begin{cases} \left( \sum_{j=1}^d \tau( |y_j|^r ) \right)^{1/r}, & r < \infty \\ \max_j(\norm{y_j}), & r = \infty. \end{cases} 
		\]
		We can thus evaluate the seminorm $\norm{\partial^k \ev_{A,\tau}(f)(x)}_{\mathcal{M}_d^k(A)}$ for each $x \in A_{\sa}^d$.  One can show this norm is finite whenever $f$ is a trace polynomial by using the non-commutative H{\"o}lder's inequality (see \cite[p. 30]{JekelLiShlyakhtenko2022}).  We finally define a universal seminorm of the $k$th derivative of $f$ by considering all possible choices of $A$ and $x$.  Let
		\[
		\norm{\partial^k f}_R = \sup \{ \norm{\partial^k \ev_{A,\tau}(f)(x)}_{\mathcal{M}_d^k(A)}: (A,\tau) \text{ tracial $\mathrm{C}^*$-algebra}, x \in A_{\sa}^d, \norm{x}_\infty \leq R \}.
		\]
		Here we use the expression ``$\partial^k f$'' in a formal sense, because we have not stated here what $\partial^k f$ abstractly; as explained in \cite{JekelLiShlyakhtenko2022}, one can define a trace polynomial $\partial^k f(x)[y_1,\dots,y_k]$ in self-adjoint $d$-tuples $x$, $y_1$, \dots, $y_k$, but we will not go into further detail here.
		
		The space of tracial non-commutative smooth functions $C_{\tr}^\infty(\R^{*d})$ is defined as the Fr{\'e}chet-space completion of the space of $d$-variable trace polynomials with respect to the family of seminorms $\norm{\partial^k f}_R$ for $R > 0$ and $k \in \N$.  Similarly, $C_{\tr}^k(\R^{*d})$ is defined as the completion with respect to $\norm{\partial^{k'} f}_R$ for $R > 0$ and $k' = 0$, \dots, $k$.
		
		The next lemma relates tracial noncommutative smooth functions with the Lip-norms from the previous section.  This lemma will also be the basis for transferring quantum metric space properties using noncommutative smooth transport from \cite{JekelLiShlyakhtenko2022}.
		
		\begin{lemma} \label{lem: smooth to Lipschitz}
			Let $(A,\tau)$ be a tracial $\mathrm{C}^*$-algebra generated by self-adjoint elements $(x_i)_{i=1}^d$ with $\norm{x_i} \leq R$.  Let $\mathcal{C}$ be a collection of triples $(\tilde{A},\alpha,\beta)$ contained in $\mathcal{C}_{\tr}(A,\tau)$.  Let $f \in C_{\tr}^1(\R^{*d})$.  Then
			\[
			L_{x,\mathcal{C}}(\ev_{A,\tau}(f)(x)) \leq \norm{\partial f}_R.
			\]
		\end{lemma}
		
		\begin{proof}
			Consider $(\tilde{A},\alpha,\beta) \in \mathcal{C}$.  By assumption $\tilde{A}$ has a faithful trace $\tilde{\tau}$.  Since $\alpha$, $\beta$ are trace-preserving, we also have
			\[
			\alpha(\ev_{A,\tau}(f)(x)) = \ev_{\tilde{A},\tilde{\tau}}(f)(\alpha(x)),
			\]
			and similarly for $\beta$ rather than $\alpha$.  We note that for all $y \in A_{\sa}^d$ with $\norm{y}_\infty \leq R$, the norm of the Fr{\'e}chet derivative $\partial \ev_{\tilde{A},\tilde{\tau}}(f)(y)$ as a map from $(A_{\sa}^d, \norm{\cdot}_\infty)$ to $(A_{\sa}^d,\norm{\cdot}_\infty)$ is bounded by $\norm{\partial f}_R$.  Therefore, $\ev_{\tilde{A},\tilde{\tau}}(f)$ is $\norm{\partial f}_R$-Lipschitz on $R$-ball of $A_{\sa}^d$, with respect to $\norm{\cdot}_\infty$.  We therefore have
			\[
			\norm{\ev_{\tilde{A},\tilde{\tau}}(f)(\alpha(x)) - \ev_{\tilde{A},\tilde{\tau}}(f)(\beta(x))} \leq \norm{\partial^k f}_R \norm{\alpha(x) - \beta(x)}_\infty = \norm{\partial^k f}_R \max_i \norm{\alpha(x_i) - \beta(x_i)}.
			\]
			Since this holds for all $(\tilde{A},\alpha,\beta) \in \mathcal{C}$, we have $L_{x,\mathcal{C}}(\ev_{\tilde{A},\tilde{\tau}}(f)(x)) \leq \norm{\partial f}_R$ as desired.
		\end{proof}
		
		We are now ready to state the free transport theorem from \cite{JekelLiShlyakhtenko2022} (see also \cite{GS2014,DGS2021} for similar theorems).  We define $BC_{\tr}^\infty(\R^{*d})$ as the collection of functions in $C_{\tr}^\infty(\R^{*d})$ such that
		\[
		\norm{\partial^k f}_\infty = \sup_{R > 0} \norm{\partial^k f} < \infty
		\]
		for each $k \in \N_0$.  We also define $\tr(C_{\tr}^\infty(\R^{*d})$ as the image of $C_{\tr}^\infty(\R^{*d})$ under the map $f \mapsto \tr(f)$; here $\tr(f)$ is defined so that for a trace polynomial $p_0 \tr(p_1) \dots \tr(p_m)$, we have $\tr(p_0 \tr(p_1) \dots \tr(p_m)) = \tr(p_0) \tr(p_1) \dots \tr(p_m)$, or so that $\ev_{A,\tau}(\tr(f))(x) = \tau(\ev_{A,\tau}(f)(x))$ (see \cite[Corollary 3.24]{JekelLiShlyakhtenko2022} for details).  We define $\tr(BC^\infty(\R^{*d})$ similarly.
		
		\begin{theorem}[{Free transport from \cite{JekelLiShlyakhtenko2022}}] \label{thm: free transport}
			Let $V \in \tr(C_{\tr}^\infty(\R^{*d}))$ be self-adjoint of the form $V(x) = \frac{1}{2} \sum_{j=1}^d \tr(x_j^2) + W(x)$ where $W \in \tr(BC_{\tr}^\infty(\R^{*d}))$.  Assume that
			\[
			\norm{\partial^2 W}_\infty < 1.
			\]
			\begin{enumerate}[(1)]
				\item For $X \in (\mathbb{M}_n)_{\sa}^d$, let $V(X) = \ev_{\mathbb{M}_n,\tr_n}(V)(X)$.  Let $\mu_V^{(n)}$ be the random matrix measure defined by \eqref{eq: matrix model from potential}, and let $X^{(n)}$ be the corresponding random matrix tuple in $(\mathbb{M}_n)_{\sa}^d$.  There exists a tracial von Neumann algebra $(M,\tau)$ generated by some $x = (x_i)_{i=1}^d$ such that
				\[
				\lim_{n \to \infty} \tr_n(p(X^{(n)})) = \tau(p(x)) \text{ almost surely}
				\]
				for all noncommutative polynomials $p$.  See \cite[Corollary 8.2]{JekelLiShlyakhtenko2022}.
				\item Let $A$ be the $\mathrm{C}^*$-algebra generated by $x$ with trace $\tau_A = \tau|_A$.  Let $z$ be a $d$-tuple of freely semcircular elements, and let $B$ be the $\mathrm{C}^*$-algebra generated by $z$ with trace $\tau_B$.  Then there exist functions $\mathbf{f} = (f_1,\dots,f_d)$ and $\mathbf{g} = (g_1,\dots,g_d)$ self-adjoint in $C_{\tr}^\infty(\R^{*d})$ and a trace-preserving isomorphism $\Phi: A \to B$ such that
				\begin{equation} \label{eq: free transport isomorphism}
					\Phi^{-1}(z_j) = \ev_{A,\tau_A}(f)(x_j), \qquad \Phi(x_j) = \ev_{B,\tau_B}(g)(z_j).
				\end{equation}
				See \cite[Observation 8.5, Corollary 8.6]{JekelLiShlyakhtenko2022}.
			\end{enumerate}
		\end{theorem}
		
		\begin{corollary} \label{cor: free Gibbs case}
			Let $V, A,\tau_A,x, B,\tau_B,z$ be as in Theorem \ref{thm: free transport}.  Let $\mathcal{C} = \mathcal{C}_{\tr}(A,\tau_A)$ be as in Example \ref{ex: trace preserving}.  Then the Lipschitz seminorm $L_{x,\mathcal{C}}$ defines a compact quantum metric space structure on $A$.
		\end{corollary}
		
		\begin{proof}
			Let $\mathcal{C}' = \mathcal{C}_{\tr}(B,\tau_B)$ be the corresponding Lipschitz seminorm for the $\mathrm{C}^*$-algebra $B$ generated by the free semicircular family $z$.  By Corollary \ref{cor: general Lip norm for q Gaussian}, we have that $L_{z,\mathcal{C}'}$ defines a compact quantum metric space structure on $B$.  Using the isomorphism $\Phi$, we see that $L_{\Phi^{-1}(z),\mathcal{C}}$ defines a compact quantum metric space structure on $(A,\tau_A)$.
			
			By \eqref{eq: free transport isomorphism}, $\Phi^{-1}(z)$ can be expressed by the evaluation of $BC_{\tr}^\infty(\R^{*d})$ functions on $x$ and vice versa.  Therefore, by Lemma \ref{lem: smooth to Lipschitz}, $L_{x,\mathcal{C}}(\Phi^{-1}(z_j))$ and $L_{\Phi^{-1}(z),\mathcal{C}}(x_j)$ are finite.  Hence, by Lemma \ref{lem: change of coordinates}, $L_{x,\mathcal{C}}$ and $L_{\Phi^{-1}(z),\mathcal{C}}$ are equivalent up to constants.  Hence, the total boundedness condition transfers from $L_{\Phi^{-1}(z),\mathcal{C}}$ to $L_{x,\mathcal{C}}$, and so $L_{x,\mathcal{C}}$ defines a compact quantum metric space structure on $A$, as claimed.
		\end{proof}
		
		\bibliographystyle{plain}
		\bibliography{freeNCG}
		
	\end{document}